\documentclass[default]{sn-jnl}

\usepackage{graphicx}%
\usepackage{multirow}%
\usepackage{amsmath,amssymb,amsfonts}%
\usepackage{amsthm}%
\usepackage{mathrsfs}%
\usepackage[title]{appendix}%
\usepackage{xcolor}%
\usepackage{textcomp}%
\usepackage{manyfoot}%
\usepackage{booktabs}%
\usepackage{listings}%
\usepackage{mathtools}
\usepackage{mathdots} 

\usepackage{calrsfs} 

\usepackage{tikz}
\usetikzlibrary{cd}

\numberwithin{equation}{section}

\newtheorem{main}{Main Theorem}

\newtheorem{theorem}{Theorem}[subsection]
\newtheorem{definition}[theorem]{Definition}
\newtheorem{proposition}[theorem]{Proposition}
\newtheorem*{proposition*}{Proposition}

\newtheorem*{corollary*}{Corollary}
\newtheorem{lemma}[theorem]{Lemma}

\newtheorem*{conjecture*}{Conjecture}

\theoremstyle{definition}
\newtheorem{claim}[theorem]{Claim}

\theoremstyle{remark}
\newtheorem{remark}[theorem]{Remark}
\newtheorem*{remark*}{Remark}

\newtheorem*{example*}{Example}

\newcommand{\defeq}{\coloneqq}

\newcommand{\C}{\mathbf{C}}

\newcommand{\Z}{\mathbf{Z}}

\DeclareMathOperator{\Ker}{Ker}

\DeclareMathOperator{\ad}{ad}
\DeclareMathOperator{\Ad}{Ad}

\newcommand{\Slice}{S}

\newcommand{\Kaz}{\mathrm{K}}

\newcommand{\Sym}{\mathcal{S}}
\newcommand{\Poi}{\mathcal{P}}
\newcommand{\Univ}{\mathcal{U}}
\newcommand{\V}{\mathcal{V}}
\newcommand{\W}{\mathcal{W}}

\DeclareMathOperator{\Aut}{Aut}
\DeclareMathOperator{\gr}{gr}
\newcommand{\Mat}{\mathrm{Mat}}
\DeclareMathOperator{\Span}{Span}

\renewcommand{\sl}{\mathfrak{sl}}
\renewcommand{\sp}{\mathfrak{sp}}
\newcommand{\so}{\mathfrak{so}}
\newcommand{\g}{\mathfrak{g}}
\newcommand{\h}{\mathfrak{h}}
\newcommand{\m}{\mathfrak{m}}
\newcommand{\p}{\mathfrak{p}}

\newcommand{\Alt}{\mathsf{\Lambda}}

\newcommand{\coH}{\mathsf{H}}

\newcommand{\module}[1]{{#1}\text{-}\mathsf{mod}}
\newcommand{\whcat}{\mathsf{W}}
\DeclareMathOperator{\Ind}{Ind}
\DeclareMathOperator{\Whi}{Wh}

\begin{document}

\title[]{Reduction by stages for finite W-algebras}


\author[1]{\fnm{Naoki} \sur{Genra}}\email{genra@ms.u-tokyo.ac.jp}

\author*[2]{\fnm{Thibault} \sur{Juillard}}\email{thibault.juillard@universite-paris-saclay.fr}

\affil[1]{\orgdiv{Graduate School of Mathematical Sciences}, \orgname{University of Tokyo 3-8-1 Komaba}, \orgaddress{ \city{Tokyo}, \postcode{153-8914}, \state{Meguro}, \country{Japan}}}

\affil*[2]{\orgdiv{Institut de mathématiques d’Orsay}, \orgname{Université Paris-Saclay, CNRS}, \orgaddress{\street{Rue Magat}, \city{Orsay}, \postcode{91405}, \country{France}}}


\abstract{Let $\mathfrak{g}$ be a simple Lie algebra: its dual space $\mathfrak{g}^*$ is a Poisson variety. It is well known that for each nilpotent element $f$ in $\mathfrak{g}$, it is possible to construct a new Poisson structure by Hamiltonian reduction which is isomorphic to some subvariety of $\mathfrak{g}^*$, the Slodowy slice $S_f$. Given two nilpotent elements $f_1$ and $f_2$ with some compatibility assumptions, we prove Hamiltonian reduction by stages: the slice $S_{f_2}$ is the Hamiltonian reduction of the slice $S_{f_1}$. We also state an analogous result in the setting of finite W-algebras, which are quantizations of Slodowy slices.  These results were conjectured by Morgan in his PhD thesis. As corollary in type A, we prove that any hook-type W-algebra can be obtained as Hamiltonian reduction from any other hook-type one. As an application, we establish a generalization of the Skryabin equivalence. Finally, we make some conjectures in the context of affine W-algebras.}




\maketitle

\section{Introduction}

\subsection{Background}

In this article, we investigate the problem of Hamiltonian reduction by stages for the Slodowy slices of a simple complex Lie algebra. In symplectic geometry, given an action of a Lie group on a symplectic manifold and a moment map, one can perform \textit{Hamiltonian reduction}, which consists in taking the quotient of the zero-fiber of the moment map by the action of the group. This quotient is then equipped with an induced symplectic structure. When the group contains a normal subgroup, this reduction can be performed in \textit{two stages}. One can first take the Hamiltonian reduction by this normal subgroup (\textit{first stage}). Then, one can consider the induced action of the quotient of the two groups on the reduced symplectic manifold and perform Hamiltonian reduction again (\textit{second stage}). This procedure is called \textit{reduction by stages}. Under good assumptions, the Hamiltonian reduction by the ambient group and the reduction by stages result in the same symplectic variety, up to a natural isomorphism \cite[Section 5.2]{marsden2007reduction}.

This theory is developed in a very general setting, for symplectic manifolds, by Marsden, Misiolek, Ortega, Perlmutter, and Ratiu in \cite{marsden2007reduction}. In the context of Slodowy slices, it was first studied by Morgan in his PhD thesis \cite{morgan2015phd, morgan2015quantum}. In a different setting, Kamnitzer, Pham and Weekes proved in \cite{kamnitzer2022hamiltonian} that there are Hamiltonian reduction relations between some generalized slices of the affine Grassmannian of a reductive algebraic group. These affine Grassmanian slices are related to Slodowy slices under the Mirkovi\'c-Vybornov isomorphism \cite{mirkovic2008quiver}.

Let us introduce the precise set-up for Slodowy slices. Let $G$ be a complex connected algebraic group and let $\g$ be its corresponding Lie algebra. We assume that $\g$ is simple and we denote by $\g^*$ the dual space, which is a Poisson variety. Let $f$ be a nilpotent element in $\g$; it can be embedded in an $\sl_2$-triple $(e, h, f)$. To this triple, one can associate a subvariety of $\g^*$, denoted by $\Slice_f$ and called a Slodowy slice. In particular, $\Slice_f$ is an affine space transversal to the coadjoint orbits. These orbits are the symplectic leaves of $\g^*$, hence one can deduce from the transversality that the Slodowy slice is also a Poisson variety. Gan and Ginzburg proved that $\Slice_f$ is isomorphic to the Hamiltonian reduction of $\g^*$ with respect to the Hamiltonian action of $M$, a unipotent subgroup of $G$ \cite{gan2002quantization}; see Section \ref{section:geometrical-reduction} for more details. Let $\mu : \g^* \rightarrow \m^*$ be the moment map, where $\m$ is the Lie algebra $M$. We denote by $\g^* /\!/ M$ the Hamiltonian reduction, which is by definition the quotient $\mu^{-1}(0)/M$. Gan and Ginzburg prove the isomorphism: $\Slice_f \cong \g^* /\!/ M$. 

Let $f_1, f_2 \in \g$ be two nilpotent elements such that the adjoint orbit of $f_1$ is contained in the Zariski closure of the orbit of $f_2$. Morgan conjectured that there exist a unipotent subgroup $M_0$ of $G$ and a Hamitonian $M_0$-action on $\Slice_{f_1}$ such that $\Slice_{f_2} \cong \Slice_{f_1}/\!/M_0$ \cite[Objective 3.6]{morgan2015quantum}. If we denote $M_i$ the unipotent group such that $\Slice_{f_i} \cong \g^* /\!/ M_i$ ($i=1,2$), one expects that $M_0 = M_2 / M_1$.  This can can be reformulated as the following isomorphism between one stage and two stages reductions:
$$ \g^*/\!/M_2 \cong (\g^*/\!/M_1)/\!/M_0. $$

The finite W-algebra $\Univ(\g, f)$ is an associative algebra constructed as the quantum Hamiltonian reduction of the universal enveloping algebra of $\g$, denoted by $\Univ(\g)$. Then $\Univ(\g, f)$ is equipped with the Kazhdan filtration and the associated graded algebra is a Poisson algebra which is isomorphic to the coordinate ring of the Slodowy slice $\Slice_{f}$, so $\Univ(\g, f)$ is a quantization of $\Slice_{f}$ \cite{gan2002quantization}; see Section \ref{section:quantum-reduction} for more details. Morgan conjectured that $\Univ(\g, f_2)$ is isomorphic to the quantum Hamiltonian reduction of $\Univ(\g, f_1)$ \cite[Objective 3.6]{morgan2015quantum}.

However, Morgan found out that stronger conditions should happen for the two nilpotent elements to construct reduction by stages: see the discussion after Objective~3.6 in \cite{morgan2015quantum}. In type A, for any pair $f_1, f_2$ of nilpotent elements whose orbits are adjacent for the closure ordering, he provides a general construction of a group $M_0$, with a Hamiltonian action on the Slodowy slice $S_{f_1}$, satisfying these conditions. He conjectured that the resulting Hamiltonian reduction $\Slice_{f_1}/\!/M_0$ is isomorphic to the Slodowy slice $S_{f_2}$, and he proved it for a particular case (subregular and regular orbits) and for more examples in small ranks: see Conjecture~3.13 and the development made after in \cite{morgan2015quantum}.

\subsection{Main results}

In the present paper, we prove reduction by stages under some technical assumptions on the pair of nilpotent elements, which are in fact a refinement of Morgan's conditions. We also provide examples of nilpotent elements for which Hamiltonian reduction by stages holds. As application, we prove a generalization of the Skryabin equivalence. 

Let $f_1, f_2$ be nilpotent elements in $\g$. For $i=1,2$, set $\Gamma^i : \g = \bigoplus_{j\in\Z}\g^{(i)}_j$ a $\Z$-grading that is good for $f_i$, and let  $(e_i, h_i, f_i)$ be a $\sl_2$-triple in $\g$ such that $e_i \in \g^{(i)}_2$ and $h_i \in \g^{(i)}_0$ ($i=1,2$). The grading $\Gamma^i$ corresponds to the eigenspace decomposition of the adjoint action of some semisimple element $q_i$. Without loss of generality, we assume that $q_1$ and $q_2$ belong to a same Cartan subalgebra $\h$ of $\g$. Using these gradings, we can construct $M_i$ unipotent subgroup of $G$ such that $\Slice_{f_i} \cong \g^*/\!/M_i$, and we denote by $\m_i$ the Lie algebra of $M_i$.

\begin{main}[Theorem \ref{theorem:freeness}]\label{main:freeness}
	We make the following assumptions:
	\begin{enumerate}
		\item there is a direct sum decomposition $\m_2 = \m_1 \oplus \m_0$, where $\m_0$ is a $\h$-stable Lie subalgebra of $\m_2$ and $\m_1$ is an ideal of $\m_2$,
		
		\item the nilpotent element $f_0 \defeq f_2 - f_1$ and the Lie algebra $\m_0$ are contained in $\g^{(1)}_0$,
		
		\item the semisimple element $q_0 \defeq q_2 - q_1$ commutes with $f_1$.
	\end{enumerate} 
	Then there exists a Hamiltonian $M_0$-action on $\Slice_{f_1} \cong \g^*/\!/M_1$, with moment map $\mu_0$, such that the action is free, in the sense of the following $M_0$-isomorphism:
	$$ M_0 \times \Slice_{f_2} \cong {\mu_0}^{-1}(0). $$
\end{main}

This first theorem is an analogue of \cite[Lemma 2.1]{gan2002quantization} (see Theorem \ref{theorem:slice-iso}). It has many useful consequences, the first being the reduction by stages.

\begin{main}[Theorem \ref{theorem:geometrical-stage-reduction}]\label{main:geometrical-stage-reduction}
	Assume the hypotheses of Main Theorem \ref{main:freeness}. Then the Hamiltonian reduction of $\Slice_{f_1}$ by the $M_0$-action is isomorphic to $\Slice_{f_2}$. In other terms: $$ \g^*/\!/M_2 \cong (\g^*/\!/M_1)/\!/M_0 .$$
\end{main}

By Proposition \ref{proposition:nilpotent-order}, the assumptions of Main Theorem \ref{main:freeness} imply that the orbit of $f_1$ is contained in the Zariski closure of the orbit of $f_2$, that is to say the assumption made by Morgan in his conjecture.

Under the assumptions of Main Theorem \ref{main:geometrical-stage-reduction}, we will prove that there exists a Lie algebra embedding $\m_0 \hookrightarrow \Univ(\g, f_1)$. Then the quantum Hamiltonian reduction of $\Univ(\g, f_1)$ is well-defined.

\begin{main}[Theorem \ref{theorem:quantum-stage-reduction}] \label{main:quantum-stage-reduction}
	Assume the hypotheses of Main Theorem \ref{main:freeness}. Then $\Univ(\g, f_2)$ is isomorphic to the quantum Hamiltonian reduction of $\Univ(g, f_1)$. 
\end{main}

Recall that nilpotent orbits in $\sl_n$ are classified by partitions of $n$. \textit{Hook-type} nilpotent elements in $\sl_n$ are nilpotent elements that belong to nilpotent orbits corresponding to a partition $(\ell, 1^{n - \ell})$, for some $1 \leqslant \ell \leqslant n$.

\begin{corollary*}[Proposition \ref{proposition:hook}]\label{corollary:hook-type-reduction}
	Let $f_1, f_2$ be hook-type nilpotent elements in $\g=\sl_n$ such that the orbit of $f_i$ corresponds to the partition $(\ell_i, 1^{n - \ell_i})$ for $i=1, 2$ with $1\leq \ell_1< \ell_2 \leq n$.
	\begin{enumerate}
		\item There exist a unipotent subgroup $M_0$ in $G$ and a Hamiltonian $M_0$-action on $\Slice_{f_1}$ such that $\Slice_{f_2} \cong \Slice_{f_1}/\!/M_0$.
		\item The W-algebra $\Univ(\sl_n, f_2)$ is isomorphic to the quantum Hamiltonian reduction of $\Univ(\sl_n, f_1)$.
	\end{enumerate}
\end{corollary*}

In addition, we have examples of reduction by stages which are not hook-type in $\sl_n$ and we also have several pairs of nilpotent elements in other types (classical and exceptional) that satisfy the assumptions of Main Theorem \ref{main:geometrical-stage-reduction}; see Section \ref{section:hook}.

As applications of Main Theorem \ref{main:quantum-stage-reduction}, we have the Skryabin equivalence by stages. Let $\whcat(\g, f_i, \m_i)$ be the category of finitely generated Whittaker $\g$-modules with respect to to $f_i$ and $\m_i$, and let $\module{\Univ(\g, f_i)}$ be the category of finitely generated $\Univ(\g, f_i)$-modules ($i=1,2$). The Skryabin equivalence says that $\whcat(\g, f_i, \m_i)$ and $\module{\Univ(\g, f_i)}$ are equivalent and the equivalence is given by an explicit functor $\Whi_i : \whcat(\g, f_i, \m_i)~\rightarrow \module{\Univ(\g, f_i)}$ \cite{skryabin2002equivalence}; see Section \ref{section:skryabin} for precise statements. Now, we suppose the assumptions of Main Theorem \ref{main:geometrical-stage-reduction}. Then there exists an embedding $\mathfrak{m}_0 \hookrightarrow \Univ(\g, f_1)$, and thus we can define a category $\whcat_0$ of finitely generated Whittaker $\Univ(\g, f_1)$-modules. 

\begin{main}[Theorem \ref{theorem:skryabin-stage}] \label{main:skryabin-stage}
	Assume the hypotheses of Main Theorem \ref{main:geometrical-stage-reduction}. Then we can construct a functor $\Whi_0 : \whcat_0 \rightarrow \module{\Univ(\g, f_2)}$ which is an equivalence of categories. Moreover, $\Whi_2 = \Whi_0 \circ \Whi_1$, where $\Whi_1$ is restricted to $\whcat(\g, f_2, \m_2) \subseteq \whcat(\g, f_1, \m_1)$.
\end{main}

Let us mention that, since the prepublication of these results, the study of the equivalence functor $\Whi_0$ has been continued by Masut in \cite{masut2024monoidal}. She proved that this equivalence is equivariant with respect to some bi-actions of a monoidal subcategory of $\module{\Univ(\g)}$ on the categories $\whcat_0$ and $\module{\Univ(\g, f_2)}$.

\subsection{Reduction by stages for affine W-algebras}

In this part, we give some conjectures and comments about reduction by stages in the setting of \textit{vertex algebras}. Roughly speaking, a vertex algebra $V$ is a non-commutative and non-associative differential algebra with some additional structures. See \cite{kac1998vertex, frenkel2004vertex} for an introduction to this subject. To any vertex algebra $V$, one can always associate a Poisson variety $X_V$, called the associated variety \cite{arakawa2012remark}.

Let $\V^k(\g)$ be the affine vertex algebra of a semisimple Lie algebra $\g$ at level $k \in \C$. It is a vertex algebra whose associated variety is $\g^*$. Then one can construct the \textit{affine W-algebra} $\W^k(\g, f)$ associated to $\g$ and a nilpotent element $f$, at level $k$. It is defined as the generalized Drinfeld-Sokolov reduction $\coH^0_f(\mathsf{L}\m, \V^k(\g))$, which is a cohomological analogue of Hamiltonian reductions for a loop algebra $\mathsf{L}\m \defeq \mathfrak{m}[t, t^{-1}]$, where $\m$ is the nilpotent Lie algebra introduced above. See \cite{feigin1992affine, kac2003quantum} for a rigorous construction of these affine W-algebras. The associated variety of $\W^k(\g, f)$ is the Hamiltonian reduction $\g^* /\!/M \cong S_f$ \cite{arakawa2015associated}. The following conjecture has been suggested in \cite[Theorem 1 and Appendix B]{MR97} for pairs of some hook-type nilpotent elements in $\g=\sl_n$.

\begin{conjecture*}
	Assume the hypotheses of Main Theorem \ref{main:freeness}. Then $\W^k(\g, f_2)$ is isomorphic to the Drinfeld-Sokolov reduction of $\W^k(\g, f_1)$: $$\W^k(\g, f_2) \cong \coH^0_{f_0}(\mathsf{L}\m_0, \W^k(\g, f_1)). $$ 
\end{conjecture*}

Even if such a result has not be proved yet, some evidences appeared recently that the reduction by stages on the W-algebras can be inverted in the following sense. We call \textit{inverse (quantum) Hamiltonian reduction} an embedding of the form $$\W^k(\g, f_1) \hookrightarrow \W^k(\g, f_2) \otimes_{\C} V, $$ for some vertex algebra $V$, under the assumptions of Main Theorem \ref{main:freeness}. Such embeddings have already been constructed very recently by Fehily in type A, for hook type nilpotent elements \cite{fehily2023hook}; and by Fasquel and Nakatsuka in type B, for subregular and regular nilpotent elements \cite{fasquel2023orthosymplectic}. See also \cite{ACG21} for other examples.

Set $f_\mathrm{reg}$ a regular nilpotent element and $f_\mathrm{sub}$ a subregular nilpotent element in $\sl_n$. The previous conjecture claims that $\W^k(\mathfrak{sl}_n, f_\mathrm{reg}) \cong \coH^0_{f_0}(\mathsf{L}\mathfrak{m}_0, \W^k(\sl_n, f_\mathrm{sub}))$. If the conjecture is true, given a $\W^k(\sl_n, f_\mathrm{sub})$-module $V_1$, $\coH^0_{f_0}(\mathsf{L}\mathfrak{m}_0, V_1)$ becomes a $\W^k(\mathfrak{sl}_n, f_\mathrm{reg})$-module. On the other hand, it is known by \cite{adamovic2019realizations, AKR21, fehily2022subregular} that there exists a vertex algebra embedding $$\psi \colon \W^k(\sl_n, f_\mathrm{sub}) \hookrightarrow \W^k(\mathfrak{sl}_n, f_\mathrm{reg}) \otimes_{\C} \Pi,$$ where $\Pi$ denotes a half lattice vertex algebra. Then, for any $\W^k(\mathfrak{sl}_n, f_\mathrm{reg})$-module $V_2$, the tensor product $V_2\otimes_{\C}\Pi$ has a structure of $\W^k(\sl_n, f_\mathrm{sub})$-module through $\psi$. This procedure is called the \textit{inverse Hamiltonian reduction}. Moreover, $\psi$ factors to the \textit{simple quotients} of these W-algebras for all $k \in \C$ such that $i(k+n-1) \notin \Z_{\geqslant1}$ with all $i\in\{1, \ldots, n-1\}$. Consequently, the inverse Hamiltonian reduction plays an important role to study representation theory of the W-algebras. 

\subsection{A conjecture of Arakawa, van Ekeren and Moreau}

Let us present some possible applications of reduction by stages for simple affine W-algebras. Set $\g = \sl_n$, a simple Lie algebra of type A. We assume that the complex number $k$ is \textit{admissible}, of the form:
$$ k = - n + \frac{n}{q}, \quad q \in \Z_{\geqslant 1}, \quad q, n \ \text{coprime}. $$ 
Denote by $\W_k(\g, f)$ the simple quotient of $\W^k(\g, f)$. Let $f$ be a nilpotent element corresponding to a partition of $n$ of the shape $(q^m, \mathsf{\lambda}')$, where $\mathsf{\lambda}'$ is a partition of $s \defeq n - qm$, and $f'$ be a nilpotent element in $\sl_s$ corresponding to the partition $\mathsf{\lambda}'$. Arakawa, van Ekeren and Moreau have conjectured the following isomorphism in \cite{arakawa2022singularities}:
$$ \W_{-n + n/q}(\sl_n, f) \cong \W_{-s + s/q}(\sl_s, f'). $$

The associated variety of the left-hand side W-algebra is given by the nilpotent Slodowy slice $\overline{O_q} \cap S_{f}$, where $ \overline{O_q}$ denotes the closure of the nilpotent orbit corresponding to the partition $(q^d, 1^r)$ and $n = d q + r$ is the euclidean division of $n$ by $q$. The associated variety of the right-hand side W-algebra is given by the nilpotent Slodowy slice $\overline{O'_q} \cap S_{f'}$, where $ \overline{O'_q}$ denotes the closure of the nilpotent orbit corresponding to the partition $(q^{d'}, 1^{r'})$ and $s = d' q + r'$ is the euclidean division of $s$ by $q$. These two varieties are known to be smoothly equivalent by the work of Kraft and Procesi \cite{kraft1981minimal}. Actually, these nilpotent Slodowy slices are isomorphic quiver varieties \cite{maffei2005quiver, henderson2015singularities}. In a future work, we will provide an interpretation of this isomorphism in term of Drinfeld-Sokolov reduction, in order to study the previous conjecture.

\subsection{Notations} Throughout this paper the ground field is $\C$, the complex number field. The ring of integers is denoted by $\Z$. For any complex affine variety $X$, its ring of functions is denoted by $\C[X]$.

\subsection{Acknowledgements} T.J. is grateful to his advisor Anne Moreau, who introduced him to this interesting problem and helped the authors to improve this paper. Both authors are grateful to Tomoyuki Arakawa and Anne Moreau for the interseting discussions they had about this project and the advice they received. They also want to thank Dylan Butson, Thomas Creutzig, Zachary Fehily, Ivan Losev, David Ridout, Lewis Topley and Daniele Valeri for useful discussions, advice and comments.

They are grateful to Elisabetta Masut for noticing a typo in the type B example. They are also very grateful to the anonymous referee who made useful comments and suggestions to improve this paper.

They are grateful to the organizers of the programs ``Thematic Program on Vertex and Chiral Algebras'' (IMPA, Rio de Janeiro, spring 2022), ``Vertex Algebras and Representation Theory'' (CIRM, Luminy, June 2022) and ``Quantum symmetries: Tensor categories, Topological quantum field theories, Vertex algebras'' (CRM, Montréal, automn 2022), during which they made significant progress on this project. 

N.G. is supported by World Premier International Research Center Initiative (WPI), MEXT, Japan and JSPS KAKENHI Grant Number JP21K20317. T.J. is supported by a grant from IMPA and a public grant as part of the Investissement d'avenir project, reference ANR-11-LABX-0056-LMH, LabEx LMH.

\setcounter{tocdepth}{1}
\tableofcontents
\setcounter{tocdepth}{2}

\section{Hamiltonian reduction for Slodowy slices}
\label{section:geometrical-reduction}

In this section, we first recall the construction of Slodowy slice as Hamiltonian reduction of the dual space of a simple Lie algebra. Most of the results were proved by Gan and Ginzburg \cite{gan2002quantization} in the case of the Dynkin grading and over the complex field $\C$. Their construction is a reinterpretation and a continuation of Premets's results, who constructed the Poisson structure on Slodowy slices as induced by the bracket of the corresponding finite W-algebra: for more details, see \cite[Proposition 6.5]{premet2002slices} and Subsection \ref{section:quantum-reduction} of the present paper. These results generalized the ones obtained by Kostant for a regular nilpotent element \cite{kostant1978whittaker}. Brundan and Goodwin noticed that this construction works in the same way if one picks any good grading instead of the Dynkin one and they proved that the choice of two different gradings leads at the end to the same Poisson structure \cite{brundan2007good}. We also state and prove Main Theorems~\ref{main:freeness} and \ref{main:geometrical-stage-reduction}, which give sufficient conditions to have reduction by stages for Slodowy slices.

\subsection{Construction of Slodowy slices} \label{sec:Slodowy}
Let $\g$ be a finite dimensional simple Lie algebra over $\C$ and $G$ a connected algebraic group whose Lie algebra is $\g$. The coordinate ring of the dual space $\g^*$ is $\C[\g^*] = \Sym(\g)$, the symmetric algebra of $\g$. There is a unique way to extend the Lie bracket on $\g$ to a Poisson bracket $\{\cdot, \cdot\}$ on $\Sym(\g)$, which is called the Kirillov-Kostant Poisson structure on $\g^*$. Notice that $\g$ can be equipped with a symmetric invariant nondegenerate bilinear form $(\cdot\mid\cdot)$ and then can be identified with $\g^*$ by a $G$-isomorphism
\begin{align*}
	\Phi \colon \g \overset{\sim}{\longrightarrow} \g^*,\quad
	x \longmapsto (x\mid\cdot).
\end{align*}
Let $\ad$ and $\Ad$ be the adjoint actions of $\g$ and $G$ on $\g$, respectively. The dual space $\g^*$ is equipped with the coadjoint actions $\ad^*$ and $\Ad^*$. Let $f$ be a nilpotent element of $\g$ and $\Gamma$ be a \textit{good grading} on $\g$ for $f$. This is a $\Z$-grading
\begin{align*}
	\Gamma: \g = \bigoplus_{j\in\Z}\g_j
\end{align*}
satisfying the following conditions: $[\g_i, \g_j] \subset \g_{i+j}$; $f \in \g_{-2}$; and $\ad(f) \colon \g_j \rightarrow \g_{j-2}$ is injective for $j\geqslant1$ and surjective for $j\leqslant1$. See \cite{elashvili2005grading, brundan2007good} for the classification of good gradings. Since $\g$ is semisimple, there exists a semisimple element $q_{\Gamma}$ in $\g$ such that for all $i \in \Z$, 
$$ \g_i = \g(q_{\Gamma}, i) \defeq \{x \in \g ~  | ~  [q_{\Gamma}, x] = i x\}.$$  

According to \cite[Lemma 1.1]{elashvili2005grading}, which is a refinement of the Jacobson-Morosov Theorem, $f$ can be embedded in an $\sl_2$-triple $(e, h, f)$ such that $ h \in\g_{0}$ and $e \in \g_{2}$, called a $\Gamma$-\textit{triple}. Set $\chi \defeq \Phi(f)$ the linear form associated to $f$. Let us consider the \textit{Slodowy slice} to $f$:
$$ \Slice_{f} \defeq \Phi(f + \Ker \ad(e)) = \chi + \Ker \ad^*(e). $$
The subspace $\g_1$ is a symplectic vector space with respect to the skew-symmetric form $\chi([\cdot, \cdot])$. Choose a Lagrangian subspace $\mathfrak{l}$ of $\g_1$. With the data of $\Gamma$, $f$ and $\mathfrak{l}$, one can build the following Lie subalgebra of $\g$:
$$ \m \defeq \mathfrak{l} \oplus \bigoplus_{i \geqslant 2} \g_i. $$
We will refer to $\m$ as a \textit{good algebra}.

\begin{remark}
	A grading $\Gamma$ is said to be \textit{even} if $\g_{i} = \{0\}$ for $i \notin 2 \Z$. Many constructions become simpler with an even grading. In the case of an even grading, $\mathfrak{l} = \g_1 = \{0\}$ and the good algebra is only determined by $\Gamma$: $ \m \defeq \bigoplus_{i \geqslant 2} \g_i. $ In type A, one can always chose an even good grading for any nilpotent element. 
\end{remark}

The good algebra $\m$ is an $\ad$-nilpotent subalgebra, its dimension is $\frac{1}{2} \dim (\Ad(G) f)$ and $\chi$ restricts to a character of $\m$, that is to say $\chi([\m, \m]) = \{0\}$. There is a unique connected unipotent subgroup of $G$, let us denote it by $M$, such that the Lie algebra of $M$ is $\m$. The group $M$ acts on $\g^*$ by the co-adjoint action. Set the $\chi$-twisted restriction map:
$$ \mu :  \g^* \longrightarrow  \m^*, \quad \xi \longmapsto (\xi - \chi)|_{\m}. $$
Since $\chi$ is a character, $\mu$ is a \textit{moment map} for this action in the sense that: $\mu$ is $M$-equivariant and the co-map $\mu^* : \C[\m^*] \rightarrow \C[\g^*]$ is a Lie algebra homomorphism which makes the following Lie algebra homomorphism diagram commute.
\begin{equation} \begin{tikzcd}[column sep = small]
		{\C[\m^*]} \arrow[rr, "\mu^*"] \arrow[rd, "\text{action}"'] & & {\C[\g^*]} \arrow[ld, "{F \mapsto \{F, \cdot\}}"] \\
		& {\operatorname{Der}\C[\g^*]} &
	\end{tikzcd} \label{equation:hamilton} \end{equation}
The arrow ``action'' denotes the map induced by differentiating the action of $M$ on $\g^*$ and $\operatorname{Der}\C[\g^*]$ is the Lie algebra of derivations of the algebra $\C[\g^*]$.

It is easy to see that:
$$ \mu^{-1}(0) = \chi + \m^{\perp}, \quad \m^{\perp} \defeq \{\xi \in \g^* ~ | ~ \xi(\m) = \{0\}\}. $$
We have the inclusions $$\Ker \ad^*(e) \subseteq \Phi\left(\g_{\geqslant 0}\right) \subseteq \m^{\perp},$$ so $\Slice_f$ is contained in $\mu^{-1}(0)$. Since $\mu$ is $M$-equivariant, $\mu^{-1}(0)$ is $M$-stable and the following theorem states that it is in fact a free $M$-variety.

\begin{theorem}[\cite{gan2002quantization}] \label{theorem:slice-iso} 
	The map
	\begin{alignat*}{2}
		\phi : & \ & M \times \Slice_{f} & \longrightarrow \mu^{-1}(0)  \\
		& & (g, \xi) & \longmapsto \Ad^*(g) \xi
	\end{alignat*}
	is a well-defined isomorphism of $M$-varieties, where $M \times \Slice_{f}$ is equipped with the action given by left multiplication on $M$.
\end{theorem}

The coordinate ring of $\mu^{-1}(0)$ is $ \C[\mu^{-1}(0)] = \Sym(\g) / I $, where $I$ is the ideal spanned by the linear polynomials $y - \chi(y)$, $y \in \m$. The adjoint action of $\m$ on $\Sym(\g) = \C[\g^*]$ descends to $\C[\mu^{-1}(0)]$ because $I$ is $\ad(\m)$-stable. 

The subspace of $\ad(\m)$-invariants, denoted by $\left(\Sym(\g) / I\right)^{\ad(\m)}$, is the set of all $X \bmod I \in \Sym(\g) / I $ such that for all  $y \in \m$, $ \{y - \chi(y), X\} \in I .$ It is an associative algebra and it is equipped with a Poisson bracket inherited from $\Sym(\g)$, defined by the following formula:
\begin{equation} \{X \bmod I, Y \bmod I\} \defeq \{X, Y\} \bmod I, \label{equation:reduced-poisson} \end{equation}
for $X \bmod I, Y \bmod I \in \left(\Sym(\g) / I\right)^{\ad(\m)}$. Let us denote by $$ \Poi(\g, f) \defeq \left(\Sym(\g)/I\right)^{\ad(\m)}$$  this Poisson algebra, which is called \textit{classical finite W-algebra} in \cite{de-sole2016structure}. 

Theorem \ref{theorem:slice-iso} implies that the restriction map $\C[\mu^{-1}(0)] \twoheadrightarrow \C[\Slice_{f}]$ (the comap of the inclusion $\Slice_f \hookrightarrow \mu^{-1}(0)$) induces an isomorphism $ \Poi(\g, f) \cong \C[\Slice_{f}]. $ Hence, through this isomorphism we endow $\Slice_{f}$ with a Poisson bracket inherited from $\g^*$. In fact, this is an example of \textit{Hamiltonian reduction}. The Slodowy slice $\Slice_{f}$ is isomorphic to the Hamiltonian reduction of $\g^*$ for the Hamiltonian action of $M$ and the moment map $\mu$: $ \Slice_f \cong  \mu^{-1}(0) / M. $

\begin{remark}
	In \cite{gan2002quantization}, the construction is in fact done with $\mathfrak{l}$ being a coisotropic subspace of $\g_1$, but it is slightly more subtle. In this case, $\chi$ does not restrict to a character and one has to perform Hamiltonian reduction with respect to a coadjoint orbit in $\m^*$. 
\end{remark}

\subsection{Statement of reduction by stages}

Let $f_i \in \g$, $i=1,2$, be two nilpotent elements and $\chi_i = \Phi(f_i)$ be the associated linear form. Let $\Gamma^i : \g = \bigoplus_{j \in \Z} \g^{(i)}_j$ be a good grading for $f_i$ and set $q_i \defeq q_{\Gamma^i}$ the corresponding semisimple element. Without loss of generality, assume that $q_1$ and $q_2$ belong to the same Cartan subalgebra $\h$ of $\g$. Let $\m_i$ be a good algebra associated to $f_i$, $M_i$ be the corresponding unipotent group and $\mu_i : \g ^* \rightarrow (\m_i)^*$ be the associated moment map. Let $I_i$ be the ideal spanned by $y - \chi_i(y)$, $y \in \m_i$. Let us assume that $(e_2, h_2, f_2)$ is a $\Gamma^2$-triple. Set  $\Slice_2 \defeq \Slice_{f_2} = \chi_2 + \Ker \ad^* (e_2)$ the corresponding Solodwy slice to $f_2$. 

\begin{theorem} \label{theorem:freeness} 
	We make the following assumptions:
	\begin{enumerate}
		\item there is a direct sum decomposition $\m_2 = \m_1 \oplus \m_0$, where $\m_0$ is a $\h$-stable Lie subalgebra of $\m_2$ and $\m_1$ is an ideal of $\m_2$,
		
		\item the nilpotent element $f_0 \defeq f_2 - f_1$ and the Lie algebra $\m_0$ are contained in $\g^{(1)}_0$,
		
		\item the semisimple element $q_0 \defeq q_2 - q_1$ commutes with $f_1$.
	\end{enumerate} 
	
	Then, if we denote by $M_0$ the unipotent subgroup of $G$ whose Lie algebra is $\m_0$, then there is a natural algebraic action of $M_0$ on the quotient ${\mu_1}^{-1}(0) / M_1$. Moreover, this action is Hamiltonian with moment map 
	\begin{alignat*}{2}
		\mu_0 : & \ & {\mu_{1}}^{-1}(0) / M_1 & \longrightarrow (\m_0)^* \\
		& &{[\xi]}  & \longmapsto (\xi - \chi_2)|_{\m_0}.
	\end{alignat*}
	In addition this action is free in the sense that there is a $M_0$-isomorphism
	$$ M_0 \times \Slice_2 \cong {\mu_0}^{-1}(0). $$
\end{theorem}

The freeness of the Hamiltonian $M_0$-action implies that the reduction by stage holds for Slodowy slice.

\begin{theorem}\label{theorem:geometrical-stage-reduction}
	We keep the assumptions of Theorem \ref{theorem:freeness}. Then, the Hamiltonian reduction ${\mu_0}^{-1}(0) / M_0$ is a well-defined Poisson variety isomorphic to $\Slice_2$. The inclusion of ${\mu_{2}}^{-1}(0)$ in ${\mu_{1}}^{-1}(0)$ induces a Poisson variety isomorphism
	$$ {\mu_2}^{-1}(0)/M_2 \cong {\mu_0}^{-1}(0) / M_0. $$
	The co-homomorphism
	\begin{alignat*}{2}
		\Psi : & \ & \left(\Poi(\g, f_1)/I_0\right)^{\ad(\m_0)} & \longrightarrow  \Poi(\g, f_2) \\
		& & (X \bmod I_1) \bmod I_0 & \longmapsto  X \bmod I_2
	\end{alignat*}
	is a Poisson algebra isomorphism, where $I_0$ the ideal of $\Poi(\g, f_1)$ spanned by 
	$$(y - \chi_2(y)) \bmod I_1, \quad y \in \m_0.$$
\end{theorem} 

\begin{remark}
	Because of the inclusion $\m_0 \subseteq \g^{(1)}_0$, to check that $\m_1$ is an ideal of $\m_2$, it is enough to check that $[\mathfrak{l}_1, \m_0] \subseteq \mathfrak{l}_1$, where $\mathfrak{l}_1$ is the Lagrangian subspace chosen to construct $\m_1$. For even gradings, we get it for free because $\mathfrak{l}_1 = \{0\}$.
\end{remark}

In order to prove Theorems \ref{theorem:freeness} and \ref{theorem:geometrical-stage-reduction}, we state some preliminary results.

\begin{lemma}\label{lemma:compatibility}
	If $f_0 = f_2 - f_1 \in \g_0^{(1)}$, then $\chi_1$ and $\chi_2$ coincide on $\m_1$. Hence, we have the commutative diagram
	$$ \begin{tikzcd}[column sep = small]
		\g^*\arrow[rr, "\mu_{2}"] \arrow[rd, "\mu_{1}"'] & & (\m_{2})^* \arrow[ld, two heads] \\
		& (\m_{1})^* &
	\end{tikzcd} $$
	and ${\mu_{2}}^{-1}(0)$ is contained in ${\mu_{1}}^{-1}(0)$.
\end{lemma}

\begin{proof}
	It follows from the orthogonality of $f_0 \in \g_0^{(1)}$ and $\m_1 \subseteq \g^{(1)}_{>0}$ with respect to the invariant bilinear form.
\end{proof}

\begin{lemma}\label{lemma:f1-m0}
	Under the hypotheses of Theorem \ref{theorem:geometrical-stage-reduction}, the Lie subalgebra $\m_0$ is contained in $\Ker \ad(f_1)$.
\end{lemma}

\begin{proof}
	This inclusion is equivalent to prove that $([x, f_1] \mid z) =0$ for all $z \in \m_0$ and $x \in \g$, because  $[\g, f_1]$ is the orthognal set of $\Ker \ad(f_1)$ with respect to the bilinear form. But $z \in \m_0 \subseteq \g_0^{(1)}$ is orthogonal to any $y \in \g_j^{(1)}$ for $j \neq 0$. So we can restrict to $x \in \g^{(1)}_{2}$. 
	
	By invariance, one has $ ([x, f_1] \mid z) =  \chi_1([z, x])$, 
	and because $[z,x] \in [\m_0, \m_1] \subseteq \m_1$, the previous lemma implies $ ([x, f_1] \mid z) =  \chi_2([z, x])$. But $\chi_2$ restricted to $\m_2$ is a character, hence $ ([x, f_1] \mid z) =  0$.
\end{proof}

\begin{lemma} \label{lemma:nice-triple}
	If the assumptions of Theorem \ref{theorem:geometrical-stage-reduction} are true, then $f_1$ lies in $\g^{(2)}_{-2}$ and it can be embedded in an $\sl_2$-triple $(e_1, h_1, f_1)$ such that $e_1 \in \g^{(2)}_2$ and $h_1 \in \g^{(2)}_0$. 
\end{lemma}

\begin{proof}
	We have $[q_0, f_1] = 0$, hence $f_1 \in \g^{(2)}_{-2}$. According to \cite[Lemma 1.1]{elashvili2005grading}, one can find a $\Gamma^1$-triple $(e'_1, h_1', f_1)$. Set $h_1' = h_1 + \sum_{j \neq 0} x_j$ where $h_1 \in \g^{(1)}_0 \cap \g^{(2)}_0$ and $x_j \in \g^{(1)}_0 \cap \g^{(2)}_j$; and set $e_1' = e_1'' + \sum_{j \neq 2} y_j$ where $e_1'' \in \g^{(1)}_2 \cap \g^{(2)}_2$ and $y_j \in \g^{(1)}_2 \cap \g^{(2)}_j$. 
 
    We have
    $$ - 2 f_1 = [h_1', f_1] = [h_1, f_1] + \sum_{j\neq0} [x_j, f_1]. $$
    On the right-hand side, $[h_1, f_1] \in \g^{(2)}_{-2}$ and $[x_j, f_1] \in \g^{(2)}_{j-2}$ for $j \neq 0$. Because $f_1 \in \g^{(2)}_{-2}$, by identification one has: $[h_1, f_1] = - 2 f_1$. We apply the same argument to the identity $ h_1' = [e_1', f_1]$: looking at the gradings on both sides, one sees that $h_1 = [e_1'', f_1]$.

    To conclude, we use an argument from \cite[Proof of (3.3.10)]{collingwood1993nilpotent}. Because $h_1$ commutes with $q_1$ and $q_2$, their three adjoint actions admit a common eigenbasis. We decompose $e_1''$ as a sum of eigenvectors for $\ad(h_1)$:
    $e_1'' = e_1 + \sum_{j \neq 2} z_j$, where $e_1 \in \g^{(1)}_2 \cap \g^{(2)}_2 \cap \g(h_1, 2)$ and $z_j \in \g^{(1)}_2 \cap \g^{(2)}_2 \cap \g(h_1, j)$. We have:
	$$ h_1 = [e_1'', f_1] = [e_1, f_1] + \sum_{j \neq 2} [z_j, f_1]. $$
	One sees that, for $j \neq 2$, $[z_j, f_1] \in \g^{(1)}_0 \cap \g^{(2)}_0 \cap \g(h_1, j-2)$, and $[e_1, f_1] \in \g^{(1)}_0 \cap \g^{(2)}_0 \cap \g(h_1, 0)$. Since $h_1 \in \g^{(1)}_0 \cap \g^{(2)}_0 \cap \g(h_1, 0)$, by identification one see that $h_1 = [e_1, f_1]$.

    We conclude that $(e_1, h_1, f_1)$ is the desired $\sl_2$-triple.
\end{proof}

\subsection{Comparison of the two nilpotent orbits}

We recall that if $f$ is nilpotent in $\g$ and if $O_f$ denotes the adjoint orbit of $f$, then $O_f$ is \textit{conical}, that is to say stable under the action of $\C^{\times}$ by multiplication. When the Main Theorems hold, we can compare $O_{f_1}$ and $O_{f_2}$.

\begin{proposition}\label{proposition:nilpotent-order}
	We assume the hypotheses Theorems \ref{theorem:freeness} and \ref{theorem:geometrical-stage-reduction} and we fix an $\sl_2$-triple $(e_1, h_1, f_1)$ as in Lemma \ref{lemma:nice-triple}. Then,  $[f_0, e_1]=0$ and $O_{f_1}$ is contained in the Zariski closure of $O_{f_2}$. 
\end{proposition}

\begin{proof}
	We have $f_0 \in  \g^{(1)}_0 \cap \g^{(2)}_{-2}$ and $e_1 \in  \g^{(1)}_2 \cap \g^{(2)}_{2}$, hence $[f_0, e_1] \in \g^{(1)}_2 \cap \g^{(2)}_{0}$. But $\g^{(1)}_2 \cap \g^{(2)}_{0} \subseteq \m_2 \cap \g^{(2)}_{0} = \{0\}$, so $[f_0, e_1]=0$.
	
	According to \cite{gan2002quantization}, the $\sl_2$-triple induces a Lie algebra homomorphism $\sl_2 \rightarrow \g$ which exponentiate to a group homomorphism $ \mathrm{SL}_2 \rightarrow G$. If we restrict to the Cartan subgroup of $\mathrm{SL}_2$, we get a homomorphism $ \gamma : \C^{\times} \rightarrow G. $ It is well-known that the action
	$$ t \cdot x \defeq t^2 \Ad(\gamma(t)) x, \quad t \in \C^{\times}, \quad x \in \g, $$
	preserves $f_1 + \Ker \ad(e_1)$, and contracts to $f_1$ in the sense that: for any $y \in \Ker \ad(e_1)$,
	$$ t \cdot (f_1 + y) = f_1 + t \cdot y \longrightarrow f_1 \quad \text{as} \quad t \to \infty. $$
	Moreover, this actions preserves any nilpotent orbit. But, by assumption, $f_2$ is in $f_1 + \Ker \ad(e_1)$, hence $f_1 = \lim\limits_{t \to \infty} (t \cdot f_2) $	lies in the closure of the orbit of $f_2$.
\end{proof}

\subsection{Proof of Theorems \ref{theorem:freeness} and \ref{theorem:geometrical-stage-reduction}} 

From now,  fix a $\Gamma^1$-triple $(e_1, h_1, f_1)$ as provided by Lemma \ref{lemma:nice-triple}. Set  $\Slice_1 \defeq \Slice_{f_1} = \chi_1 + \Ker \ad^* (e_1)$ the corresponding Solodwy slice of $f_1$. For $i=1,2$, we denote by
$$ \phi_i : M_{i} \times \Slice_{i}  \overset{\sim}{\longrightarrow}  {\mu_{i}}^{-1}(0) $$
the isomorphisms given by Theorem \ref{theorem:slice-iso}. 

\begin{claim} 
	There is a natural algebraic action of $M_0$ on ${\mu_1}^{-1}(0) / M_1$ which is Hamiltonian with moment map $\mu_0$.
\end{claim}

\begin{proof}
	Let us denote by $[\xi]$ the class of $\xi$ in the quotient ${\mu_1}^{-1}(0)/M_1$. The subgroup $M_{1}$ is normal in $M_{2}$, hence the adjoint action of $M_{2}$ restricts to a well-defined action on $\m_{1}$ and one can deduce that the moment map $\mu_{1} : \g^* \rightarrow (\m_{1})^*$ is $M_{2}$-equivariant. In particular, the action of $M_{0}$ on $\g^*$ restricts to a well-defined action on ${\mu_{1}}^{-1}(0)$. Because of the normality of $M_1$ in $M_2$, the following induced action is well-defined on the quotient: for $g \in M_0$ and $\xi \in {\mu_{1}}^{-1}(0)$, set
	$$ g \cdot [\xi] \defeq [\Ad^*(g) \xi]. $$
	
	The map $\mu_0 : {[\xi]} \mapsto (\xi - \chi_2)|_{\m_0}$ is well-defined because of Lemma \ref{lemma:compatibility} and the following fact \cite[Lemma 5.2.3]{marsden2007reduction}, which is also easy to prove: for all $g \in M_{1}$ and for all $y \in \m_{2}$, one has $ y - \Ad(g) y \in \m_{1}. $
	
	It is not difficult to see that the map is $M_0$-equivariant and its co-homomorphism makes commute an analogue of the diagram (\ref{equation:hamilton}), so the action is Hamiltonian.
\end{proof}

Set
$$ \widetilde{\mu_0} :  \Slice_1  \longrightarrow (\m_0)^*, \quad \xi \longmapsto (\xi - \chi_2)|_{\m_0}. $$
Clearly the following triangle commutes:
$$ \begin{tikzcd}[row sep = small]
	\Slice_1 \arrow[dd, "\sim"'] \arrow[rd, "\widetilde{\mu_0}"] & \\
	& (\m_0)^* \\
	{\mu_1}^{-1}(0) / M_1 \arrow[ru, "\mu_0"'] & 
\end{tikzcd} $$

As a consequence of Proposition \ref{proposition:nilpotent-order},  $\chi_2 \in \Slice_1$ and by definition, $\chi_2 \in \widetilde{\mu_0}^{-1}(0)$. An easy computation gives the following claim.

\begin{claim}\label{claim:stage-moment}
	The subvariety $\widetilde{\mu_0}^{-1}(0)$ is the affine space $ \chi_2 + \Ker \ad^*(e_1) \cap {\m_0}^{\perp}. $
\end{claim}

According to Lemma \ref{lemma:f1-m0}, we have $\m_0 \subseteq \Ker \ad(f_1)$. Since $(\cdot \mid \cdot)$ induces a nondegenerate pairing $ \Ker \ad(f_1) \times \Ker \ad(e_1) \rightarrow \C$,
we deduce that the differential
$$\mathrm{d} \widetilde{\mu_0} :  \Ker \ad^*(e_1)  \longrightarrow (\m_0)^*, \quad \xi \longmapsto \xi|_{\m_0} $$
is surjective.

\begin{claim} \label{claim:isomorphism}
	The isomorphism $\phi_1 : M_{1} \times \Slice_{1} \rightarrow  {\mu_1}^{-1}(0) $ restricts to an isomorphism 
	$$ M_{1} \times \widetilde{\mu_0}^{-1}(0) \cong {\mu_{2}}^{-1}(0). $$
	Hence, the inclusion ${\mu_{2}}^{-1}(0) \hookrightarrow {\mu_{1}}^{-1}(0)$ induces an isomorphism:
	$$  {\mu_{2}}^{-1}(0) / M_1  \cong {\mu_0}^{-1}(0). $$
\end{claim}

\begin{proof}
	We first check that
	$$ \phi_1\left(M_{1} \times \widetilde{\mu_0}^{-1}(0)\right) \subseteq {\mu_{2}}^{-1}(0). $$
	Since ${\mu_2}^{-1}(0)$ is $M_1$-stable, it is enough to check that $\widetilde{\mu_0}^{-1}(0) \subseteq {\mu_{2}}^{-1}(0)$. Let $\xi \in \widetilde{\mu_0}^{-1}(0)$, $y_1\in \m_1$ and $y_0 \in \m_0$:
	$$ \mu_2(\xi)(y_1 + y_0)  = \mu_2(\xi)(y_1) + \mu_2(\xi)(y_0). $$
	We have $\mu_2(\xi)(y_1) = \mu_1(\xi)(y_1)$ and it is zero because $\xi \in {\mu_{1}}^{-1}(0)$. We have also $\mu_2(\xi)(y_0) = \widetilde{\mu_0}(\xi)(y_0)$ and it is zero because $\xi \in \widetilde{\mu_0}^{-1}(0)$.
	Hence $\mu_2(\xi)(y_1 + y_0)  = 0$ and the restriction is well-defined.
	
	We have an inclusion of two smooth irreducible closed varieties: 
	$$\phi_1\left(M_{1} \times \widetilde{\mu_0}^{-1}(0)\right) \subseteq {\mu_{2}}^{-1}(0).$$ 
	The rank formula applied to $\mathrm{d}\widetilde{\mu_0}$ gives
	$$ \dim \Ker \ad^*(e_1) = \dim(\Ker \ad^*(e_1)\cap {\m_0}^{\perp}) + \dim \m_0, $$
	and we know that $\dim \g - \dim \Ker \ad^*(e_1) = 2 \dim \m_1$, whence $\phi_1\left(M_{1} \times \widetilde{\mu_0}^{-1}(0)\right)$ and ${\mu_{2}}^{-1}(0)$ have the same dimension, so they coincide. Whence, $ M_{1} \times \widetilde{\mu_0}^{-1}(0) \cong {\mu_{2}}^{-1}(0)$. The last part of the claim follows easily.
\end{proof}

\begin{remark}
	The differential of $\phi_1$ is an isomorphism, so we can deduce the following decomposition:
	$$ \ad^*(\m_1) \chi_2 \oplus \left(\Ker \ad^*(e_1)\cap {\m_0}^{\perp}\right) = {\m_2}^{\perp}. $$
\end{remark}

Let us consider the $M_{1}$-equivariant isomorphism:
\begin{alignat*}{2}
	\phi_1 : & \ & M_{1} \times M_0 \times \Slice_{2} & \overset{\sim}{\longrightarrow} {\mu_{2}}^{-1}(0)  \\
	& & (g, g',\xi) & \longmapsto \Ad^*(gg') \xi.
\end{alignat*}
Taking quotients by $M_1$ we get the following claim. 

\begin{claim} \label{claim:freeness}
	The map
	$$\phi_0 :  M_0 \times \Slice_{2} \longrightarrow {\mu_0}^{-1}(0), \quad (g, \xi) \longmapsto g \cdot [\xi] $$
	is an isomorphism of $M_0$-varieties, where $M_0 \times \Slice_{2}$ is equipped with the action given by left multiplication on $M_0$. 
\end{claim}

This achieves the proof of Theorem~\ref{theorem:freeness}. We turn to the proof of Theorem~\ref{theorem:geometrical-stage-reduction}. As a consequence of the previous claims, the quotient ${\mu_0}^{-1}(0) / M_0$ is a variety isomorphic to $\Slice_2 \cong {\mu_{2}}^{-1}(0) / M_2$.  Its coordinate ring is
$$ \C\left[{\mu_0}^{-1}(0) / M_0\right] = \left(\Poi(\g, f_1) / I_0\right)^{\ad(\m_0)}, $$
where $I_0$ the ideal of $\Poi(\g, f_1)$ spanned by $(y - \chi_2(y)) \bmod I_1$, $y \in \m_0$. Hence it inherits a Poisson structure from $\Poi(\g, f_1)$ in the same way as Formula (\ref{equation:reduced-poisson}).

We have established an isomorphism
$${\mu_2}^{-1}(0) / M_2 \overset{\sim}{\longrightarrow} {\mu_0}^{-1}(0) / M_0 $$
and its cohomomorphism is given by the following map:
\begin{alignat*}{2}
	\Psi : & \ & \left(\Poi(\g, f_1)/I_0\right)^{\ad(\m_0)} & \longrightarrow  \Poi(\g, f_2) \\
	& & (X \bmod I_1) \bmod I_0 & \longmapsto  X \bmod I_2
\end{alignat*}
It clearly respects the Poisson structures given by Formula (\ref{equation:reduced-poisson}). The proof of Theorem \ref{theorem:geometrical-stage-reduction} is now complete.

\section{Hamiltonian reduction for finite W-algebras}
\label{section:quantum-reduction}

In this section, we recall the construction of finite W-algebras and their relation with the Slodowy slices. This was first studied by Kostant for a regular nilpotent element \cite{kostant1978whittaker}, and then by Premet for any nilpotent element \cite{premet2002slices}. Gan and Ginsburg proved that finite W-algebras are a quantization of classical W-algebras \cite{gan2002quantization} in the case of Dynkin gradings and over the complex number field. Doing so, they proved that the Poisson structure exhbited by Premet in \cite[Proposition 6.5]{premet2002slices} and the one defined by Hamiltonian reduction are the same. This was generalized to the setting of good gradings by Brundan and Goodwin in \cite{brundan2007good}, and by Sadaka to \textit{admissible gradings} \cite{sadaka2016paires}. 

Then, we state Main Theorem \ref{main:quantum-stage-reduction}, which is a reduction by stages for finite W-algebras. We introduce Kazhdan filtrations on finite W-algebras. Using these tools, we can deduce that the reduction by stages for finite W-algebras follows from the reduction by stages for Slodowy slices (see Theorem \ref{theorem:gr-two-stages}).

\subsection{Finite W-algebras} 

Let $\Univ(\g)$ be the universal enveloping algebra of the simple Lie algebra $\g$. Let $f$ a nilpotent element of $\g$ , $\Gamma$ be a good grading for $f$ and $q_{\Gamma}$ be the corresponding semisimple element. The grading $\Gamma$ induces a filtration on $\Univ(\g)$, the \textit{Kazhdan filtration}, which is defined by:
$$ \Univ_n(\g) \defeq \sum_{2i - j \leqslant n} \{ x_1 \cdots x_i \in \Univ(\g) ~ | ~ x_1, \dots, x_i \in \g, ~  [q_{\Gamma}, x_1 \cdots x_i] = j x_1 \cdots x_i\},$$
for $n \in \Z$. 
The natural embedding $\g \hookrightarrow \Univ(\g)$ induces for all $j \in \Z$ an embedding $\g_j \hookrightarrow \Univ_{2 - j}(\g). $

For all $j, k \in \Z$, $[\g_j, \g_k] \subseteq \g_{j+k}$ and $\g$ spans $\Univ(\g)$ as an algebra. Hence, $\Univ(\g)$ is \textit{almost commutative} for the Kazhdan filtration, that is to say: for all $ m,n \in \Z$,
$$  [\Univ_m(\g), \Univ_n(\g)] \subseteq \Univ_{m + n -2}(\g). $$
Hence the \textit{associated graded algebra} 
$$\gr_{\bullet} \Univ(\g) \defeq \bigoplus_{n \in \Z} \Univ_n(\g) / \Univ_{n-1}(\g)$$ 
is a commutative algebra. The image of $X \in \Univ_n(\g)$ in the quotient $\Univ_n(\g) / \Univ_{n-1}(\g)$ is denoted by $[X]_n$. Moreover, one can define a Poisson structure on $\gr_{\bullet} \Univ(\g)$ with the following Poisson bracket:
$$ \left\{[X]_m, [Y]_n\right\} \defeq [XY - YX]_{m + n -2}, $$
for $X \in \Univ_m(\g), Y \in \Univ_n(\g)$. 

The grading $\Gamma$ induces a $\C^{\times}$-action on $\g^*$. Set $ \gamma : \C^{\times} \rightarrow \Aut_{\C}(\g)$ the action defined for all $t \in \C^{\times}$, $i \in \Z$ and all $\xi \in \Phi(\g_i)$ by: $\rho(t) \xi \defeq t^{2+i} \xi.  $
The \textit{Kazhdan grading} on $\Sym(\g) = \C[\g^*]$ is defined as: for all $n \in \Z$, $\Sym_n(\g)$ is the set of $X \in \Sym(\g)$ such that for all $t \in \C^{\times}$, $ \rho^*(t) X = t^{-n} X. $ It is easy to see that
$$\Sym_n(\g) = \bigoplus_{2i - j = n} \{ x_1 \cdots x_i \in \Sym(\g) ~ | ~ x_1, \dots, x_i \in \g, ~  [q_{\Gamma}, x_1 \cdots x_i] = j x_1 \cdots x_i \}.$$

Let $\m$ be a good algebra associated to $\Gamma$ and $f$,  and $\mu : \g^* \rightarrow \m^*$ be the $\chi$-twisted restriction map. The linear form $\chi$ is  fixed by $\rho$. As a consequence, the Slodowy slice $\Slice_{f}$ and the fiber $\mu^{-1}(0)$ are $\rho$-stable affine subspaces of $\g^*$. So, there is an induced Kazhdan grading on their coordinate rings. It appears that $\Univ(\g)$ is a \textit{quantization} of the Poisson algebra $\Sym(\g) $ in the sense of the following proposition.

\begin{proposition}[\cite{gan2002quantization}] \label{proposition:quantization-Sg}
	The Lie algebra inclusion $\g \hookrightarrow \gr\Univ(\g)$, given by the mapping $\g_j \ni x \mapsto  [x]_{2 - j}$,
	induces a graded Poisson algebra isomorphism
	$$ \Sym_{\bullet}(\g) \cong \gr_{\bullet}\Univ(\g). $$
\end{proposition}

The restriction $\chi|_{\m}$ is a character, hence we can consider the associated one-dimensional representation of $\m$, denoted by $\C_{\chi}$. Then, we can consider the induced $\g$-module, called the \textit{generalized Gelfand-Graev module} associated to $\g$ and $f$:
$$ Q_f \defeq \Univ(\g) / \widehat{I} \cong \Univ(\g) \otimes_{\Univ(\m)} \C_{\chi}, $$
where $\widehat{I}$ is the left $\Univ(\g)$ ideal spanned by $y - \chi(y) $, $ y \in \m$. The Kazhdan filtration on $Q_f$ is defined thanks to the natural projection $\Univ(\g) \twoheadrightarrow Q_f$:
$$ \left(Q_{f}\right)_n \defeq \left\{ X\bmod \widehat{I} ~  : ~  X \in \Univ_n(\g)\right\}. $$

The adjoint action of $\m$ on $\Univ(\g)$ descends to $Q_f$ because $\widehat{I}$ is $\ad(\m)$-stable. Let ${Q_f}^{\ad(\m)}$ be the subspace of $\ad(\m)$-invariant elements: it is the set of all $X \bmod \widehat{I} \in Q_f$ such that for all $y \in \m$, $ [y - \chi(y), X] = [y, X] \in \widehat{I}.$ The subspace ${Q_f}^{\ad(\m)}$ is an algebra for the product induced by $\Univ(\g)$:
$$ \left(X \bmod \widehat{I}\right) \cdot \left(Y \bmod \widehat{I}\right) \defeq X Y \bmod \widehat{I} $$
for all $X + \widehat{I}, Y + \widehat{I} \in {Q_f}^{\ad(\m)}$. 

\begin{definition}[\cite{premet2002slices}]
	The \textbf{finite W-algebra} associated to $\g$ and $f$ is defined as the associated $\C$-algebra:
	$$ \Univ(\g, f) \defeq {Q_f}^{\ad(\m)} = \left(\Univ(\g) / \widehat{I}\right)^{\ad(\m)}. $$
\end{definition}

The Kazhdan filtration on $\Univ(\g, f)$ is defined by intersection:
$$ \Univ_n(\g, f) \defeq \left(Q_{f}\right)_n \cap \Univ(\g, f). $$
It is possible to prove that Hamiltonian reduction and quantization are compatible in the sense of the following theorem.

\begin{theorem}[{\cite[Theorem 4.1]{gan2002quantization}}] \label{theorem:quantization-slodowy}
	There is a graded algebra isomorphism
	$$ \gr_{\bullet} Q_f \cong (\Sym(\g)/I)_{\bullet} $$
	which make the following diagram commute.
	$$ \begin{tikzcd}
		\Sym_{\bullet}(\g) \arrow[r, "\sim"] \arrow[d, two heads] & \gr_{\bullet} \Univ(\g) \arrow[d, two heads] \\
		(\Sym(\g)/I)_{\bullet} \arrow[r, "\sim"] & \gr_{\bullet} Q_f \\
		\Poi_{\bullet}(\g, f)\ \arrow[r, "\sim"] \arrow[u, hook] & \gr_{\bullet} \Univ(\g, f) \arrow[u, hook]
	\end{tikzcd} $$
	Moreover, the isomorphism $ \Poi_{\bullet}(\g, f) \cong \gr_{\bullet} \Univ(\g, f) $ is compatible with the Poisson structures.
\end{theorem}

\subsection{Satement of reduction by stages} \label{subsec:quantum-stage}

For $i=1,2$, let $\Gamma^i : \g = \bigoplus_{j \in \Z} \g^{(i)}_j$ be a good grading for a nilpotent element $f_i$, set $q_i \defeq q_{\Gamma^i}$ the corresponding semisimple element inside a common Cartan subalgebra $\h$. Let $Q_i = Q_{f_i}$ be the associated generalized Gelfand-Graev module and set the finite W-algebra
$$ \Univ(\g, f_i) = \left(\Univ(\g) / \widehat{I_i}\right)^{\ad(\m_i)}, $$
where $\widehat{I_i}$ is the left $\Univ(\g)$-ideal spanned by $y - \chi_i(y)$, $y \in \m_i$. Assuming a decomposition $$M_2 = M_1 \rtimes M_0, $$ we would like to perform a quantum version of the reduction by stages.

\begin{lemma} \label{lemma:m0-embedding}
	Assume that: $\m_1$ is an ideal of $\m_2$, there is a nilpotent subalgebra $\m_0 \subseteq \m_2$ such that $\m_2 = \m_1 \oplus \m_0$ and $f_2 - f_1 \in \g^{(1)}_0$. Then the map
	$$ \m_0 \longrightarrow \Univ(\g, f_1), \quad y \longmapsto y \bmod \widehat{I_1} $$ 
	is an embedding of Lie algebras.
\end{lemma}

\begin{proof}
	We just need to check that for all $ y \in \m_0$ and $z \in \m_1$: $ [z,y] \in \widehat{I_1}. $	We have $[z,y] \in \m_1$ (because $\m_1$ is an ideal) so it make sense to compute:
	$$ \chi_1([z,y]) = \chi_2([z,y]) = [\chi_2(z), \chi_2(y)] = 0 $$
	where the first equality is given by Lemma \ref{lemma:compatibility}. Hence:
	$$ [z,y] = [z,y] - \chi_1([z,y]) \in \widehat{I_1}. $$
	The lemma follows.
\end{proof}

As a consequence of this lemma, we can define $\widehat{I_0}$, the left $\Univ(\g, f_1)$-ideal spanned by $\left(y - \chi_2(y)\right) \bmod \widehat{I_1}$, $ y \in \m_0$. We denote the quotient by $$ Q_0 \defeq \Univ(\g, f_1) / \widehat{I_0}. $$ 
The adjoint action of $\m_0$ on $\Univ(\g, f_1)$ descends to $Q_0$, so it makes sense to consider the subspace of invariants ${Q_0}^{\ad(\m_0)}$. We can define a product law on it in the same way it is defined for finite W-algebras.

\begin{lemma}
	The subspace ${Q_0}^{\ad(\m_0)}$ is an algebra for the product induced by $\Univ(\g, f_1)$:
	$$ \left(X \bmod \widehat{I_{0}}\right) \cdot \left(Y \bmod \widehat{I_{0}}\right) \defeq X Y \bmod \widehat{I_{0}} $$
	for $X \bmod \widehat{I_{0}}, Y \bmod \widehat{I_{0}}\in {Q_0}^{\ad(\m_0)}$.
\end{lemma}

This algebra is denoted by $ \Univ_0 \defeq \left(\Univ(g, f_1) / \widehat{I_0}\right)^{\ad(\m_0)}. $ It is in fact isomorphic to $\Univ(\g, f_2)$.

\begin{theorem} \label{theorem:quantum-stage-reduction}
	If the assumptions of Theorems \ref{theorem:freeness} and \ref{theorem:geometrical-stage-reduction} hold, then the map
	\begin{align*}
		\widehat{\Psi} :  \Univ_0 & \longrightarrow \Univ(\g, f_2) \\
		\left(X \bmod \widehat{I_1}\right) \bmod \widehat{I_0}  & \longmapsto X \bmod \widehat{I_2}.
	\end{align*}
	is an algebra isomorphism.
\end{theorem}

To prove Theorem \ref{theorem:quantum-stage-reduction}, the idea is to use the Poisson isomorphism given by Theorem \ref{theorem:geometrical-stage-reduction}. We need to endow our objects with relevant filtrations or gradings. We notice that $f_1\in \g^{(1)}_2 \cap \g^{(1)}_2$ because $[q_0, f_1] = 0$. Hence, $\chi_1 = \Phi(f_1)$ is fixed for the action $\rho_2 : \C^{\times} \rightarrow \Aut_{\C}(\g^*)$ induced by $\Gamma^2$. Hence all the following varieties of Section \ref{section:geometrical-reduction} are $\rho_2$-stable:
$$ {\mu_1}^{-1}(0), \quad \Slice_1, \quad \widetilde{\mu_0}^{-1}(0), \quad {\mu_2}^{-1}(0), \quad \Slice_2. $$
So, natural filtrations and gradings for this problem are induced by $\Gamma^2$. 

We denote by $\Univ^{(i)}_{\bullet}(\g)$ the Kazhdan filtration on $\Univ(\g)$ associated to $\Gamma^i$, for $i=1,2$. We denote the associated graded algebra by $\gr^{(i)}_{\bullet}\Univ(\g)$. By projection or intersection, the filtration $\Univ^{(2)}_{\bullet}(\g)$ associated to $\Gamma^2$ descends to: 
$$ Q_1, \quad \Univ(\g, f_1), \quad Q_0, \quad \Univ_0, \quad Q_2, \quad \Univ(\g, f_2). $$

\begin{remark}
	Since $q_2 = q_1 + q_0$, one can define the Kazhdan filtration associated to $\Gamma^2$ thanks to the one associated to $\Gamma^1$ and the semisimple element $q_0$:
	$$ \Univ^{(2)}_n(\g) = \sum_{m - \omega = n} \{X \in \Univ^{(1)}_m(\g) ~  | ~  [q_0, X]= \omega X\}. $$
\end{remark}

\begin{figure}[t]
	$$\begin{tikzcd}
		\Sym(\g) \arrow[d, two heads] \arrow[r, equal] & \Sym(\g) \arrow[d, two heads]   \\
		\Sym(\g)/I_1  \arrow[r, two heads] & \Sym(\g)/I_2  \\
		\Poi(\g, f_1) \arrow[d, two heads] \arrow[r, two heads] \arrow[u, hook] & {\left(\Sym(\g)/I_2\right)^{\ad(\m_1)}} \arrow[u, hook]  \\
		{\Poi(\g, f_1)/I_0}  \arrow[ru, "\sim"'] & \\
		\left(\Poi(\g, f_1)/I_0\right)^{\ad(\m_0)} \arrow[u, hook] \arrow[r, "\Psi", "\sim"'] & \Poi(\g, f_2) \arrow[uu, hook] \\
	\end{tikzcd} \quad \begin{tikzcd}
		\Univ(\g) \arrow[r, equal] \arrow[d, two heads] & \Univ(\g) \arrow[d, two heads] \\
		Q_1 \arrow[r, "(a)"] & Q_2  \\
		\Univ(\g, f_1) \arrow[d, two heads] \arrow[r, "(b)"] \arrow[u, hook] & {Q_2}^{\ad(\m_1)} \arrow[u, hook]  \\
		Q_0  \arrow[ru, "(c)"'] & \\
		\Univ_0 \arrow[u, hook] \arrow[r, "\widehat{\Psi}"] & \Univ(\g, f_2)\arrow[uu, hook] \\
	\end{tikzcd}$$
	\caption{Construction of $\Psi$ and $\widehat{\Psi}$}
	\label{figure:diagram-quantum}
\end{figure}
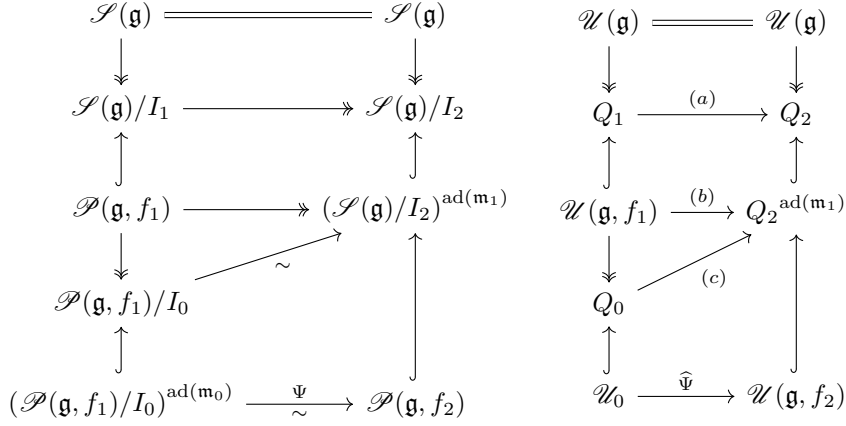

\begin{lemma}
	The map $\widehat{\Psi} : \Univ_0  \rightarrow \Univ(\g, f_2)$ of Theorem \ref{theorem:quantum-stage-reduction} is a well-defined filtered algebra homomorphism for the Kazdhan filtrations associated to $\Gamma^2$.
\end{lemma}

\begin{proof}
	To construct $\widehat{\Psi}$, we use the right diagram of Figure \ref{figure:diagram-quantum}. The arrow $(a)$ is well-defined because $\widehat{I_{1}} \subseteq \widehat{I_{2}}$, as a straightforward consequence of Lemma \ref{lemma:compatibility}. One can check it is $\m_1$-equivariant so the map $(b)$ is well-defined. $(b)$ sends $\widehat{I_0}$ into $\widehat{I_2}$ by definition of these ideals, hence $(c)$ is well-defined. It is $\m_0$-equivariant, hence $\widehat{\Psi}$ is indeed well-defined.
\end{proof}

The left diagram of Figure \ref{figure:diagram-quantum} summarizes the construction of $\Psi$ in the proof of Theorem \ref{theorem:geometrical-stage-reduction}.

\begin{proposition} \label{proposition:graded-diagram}
	The left diagram of Figure \ref{figure:diagram-quantum} consists of graded algebra homomorphisms for the Kazhdan gradings associated to $\Gamma^2$.
\end{proposition}

This proposition follows from the following lemma and the proof of Theorem~\ref{theorem:geometrical-stage-reduction}.

\begin{lemma}
	The isomorphisms
	\begin{alignat*}{3} 
		\phi_i : & \ & M_i \times \Slice_{i} & \longrightarrow {\mu_i}^{-1}(0)& &  \\
		& & (g, \xi) & \longmapsto \Ad^*(g) \xi, & & \quad i=1,2,
	\end{alignat*} 
	given in Theorem \ref{theorem:slice-iso} are $\C^{\times}$-equivariant.
\end{lemma}

This lemma is proved in \cite[Lemma 2.1]{gan2002quantization} in the case of $\Gamma$ being a Dynkin grading, by considering the map $\mathrm{SL}_2 \rightarrow G$ induced by the $\sl_2$-triple. We rephrase the proof below for any good grading.

\begin{proof}
	The subalgebra $\m_i$ is $\Z$-graded with the gradation induced by $\Gamma^2$. The scalar $t \in \C^{\times}$ acts on the homogeneous elements of $(\m_i)^{(2)}_j$, $j \in \Z$, by multiplication by $t^{-j}$.
	The group $M_i$ is unipotent so  the exponential induces an isomorphism of varieties $\exp : \m_i \overset{\sim}{\rightarrow} M_i. $ $M_i$ is equipped with a $\C^{\times}$-action by transporting the one of $\m_i$.
	
	The lemma follows if we prove that the map
	$$ M_i \times \g^* \longrightarrow \g^*, \quad (g, \xi) \longmapsto \Ad^*(g) \xi $$
	is $\C^{\times}$-equivariant. Let $z \in (\m_i)^{(2)}_j$, $j \in \Z$, and $\xi \in \Phi(\g_k^{(2)})$, $k \in \Z$. By property of the exponential:
	$$ \Ad^*(\exp(z)) \xi = \exp(\ad^*(z)) \xi = \sum_{p=0}^{\infty} \frac{1}{p!} (\ad^*(z))^p \xi. $$
	Applying the $\C^{\times}$-action one gets:
	$$ t \cdot (\Ad^*(\exp(z)) \xi) = \sum_{p=0}^{\infty} \frac{1}{p!} t^{2 - k - pj} (\ad^*(z))^p \xi.  $$
	This is obviously equal to
	$$  \sum_{p=0}^{\infty} \frac{1}{p!}  (\ad^*(t^{-j} z))^p (t^{2 - k}\xi) = \Ad^*(t \cdot \exp(z)) (t \cdot \xi), $$
	so the lemma is proved.
\end{proof}

The following theorem states that $\widehat{\Psi}$ quantizes $\Psi$.

\begin{theorem} \label{theorem:gr-two-stages}
	There are graded algebra isomorphisms
	$$ \left((\Poi(\g, f_1) / I_0)^{\ad(\m_0)}\right)^{(2)}_{\bullet} \cong \gr^{(2)}_{\bullet} \Univ_0 \quad \text{and} \quad \Poi^{(2)}_{\bullet}(\g, f_2)  \cong \gr^{(2)}_{\bullet} \Univ(\g, f_2) $$
	such that the following diagram commutes.
	$$ \begin{tikzcd}
		\left((\Poi(\g, f_1) / I_0)^{\ad(\m_0)}\right)^{(2)}_{\bullet} \arrow[r, "\Psi", "\sim"'] \arrow[d, "\sim"'] & \Poi^{(2)}_{\bullet}(\g, f_2) \arrow[d, "\sim"] \\
		\gr^{(2)}_{\bullet} \Univ_0 \arrow[r, "\gr \widehat{\Psi}"'] &  \gr^{(2)}_{\bullet} \Univ(\g, f_2)
	\end{tikzcd} $$
	Hence, the homomorphism
	$$\gr \widehat{\Psi} : \gr^{(2)}_{\bullet} \Univ_0 \longrightarrow  \gr^{(2)}_{\bullet} \Univ(\g, f_2)$$
	is an isomorphism. 
\end{theorem}

Theorem \ref{theorem:gr-two-stages} implies Theorem \ref{theorem:quantum-stage-reduction}. Indeed, if it is true, $\Psi $ and $\gr \widehat{\Psi}$ are isomorphisms. Hence, according to \cite[Lemma 6.8]{kamnitzer2022hamiltonian}, $\widehat{\Psi}$ is an isomorphism because the filtrations are nonnegative and exhaustive. So, let us prove Theorem~\ref{theorem:gr-two-stages} to achieve the proof of Main Theorem~\ref{main:quantum-stage-reduction}. 

\subsection{Proof of Theorem \ref{theorem:gr-two-stages}} We split the proof in three propositions. The first one says that Theorem \ref{theorem:quantization-slodowy} remains true for the W-algebra $\Univ(\g, f_1)$ if we we use the Kazhdan filtration associated to $\Gamma^2$ instead of $\Gamma^1$.

\begin{proposition} \label{proposition:blue-diagram}
	There is a graded algebra isomorphism
	$$ \Sym^{(2)}_{\bullet}(\g) \cong \gr^{(2)}_{\bullet} \Univ(\g) $$
	which makes the following diagram commutes.
	$$\begin{tikzcd}
		\Sym^{(2)}_{\bullet}(\g) \arrow[r, "\sim"] \arrow[d, two heads] & \gr^{(2)}_{\bullet} \Univ(\g) \arrow[d, two heads] \\
		(\Sym(\g)/I_1)^{(2)}_{\bullet} \arrow[r, "\sim"] & \gr^{(2)}_{\bullet} Q_1 \\
		\Poi^{(2)}_{\bullet}(\g, f_1) \arrow[r, "\sim"] \arrow[u, hook] & \gr^{(2)}_{\bullet} \Univ(\g, f_1) \arrow[u, hook]
	\end{tikzcd} $$
	
\end{proposition}

\begin{proof}
	The first line of this diagram is a composition of (not-graded) Poisson algebra isomorphisms
	$$ \Sym(\g) \overset{\sim}{\longrightarrow} \gr^{(1)} \Univ(\g) \overset{\sim}{\longrightarrow} \gr^{(2)} \Univ(\g) $$
	given by the mapping 
	\begin{equation} \g^{(1)}_j \cap\g^{(2)}_k \ni x \longmapsto [x]^{(1)}_{2-j}  \longmapsto [x]^{(2)}_{2-k} \label{equation:change-grading} \end{equation}
	for $j,k \in \Z$. It is well-defined because $[q_1, q_2]=0$.
	
	\begin{claim} \label{claim:gr-ideals}
		Under these isomorphisms, $I_1$ is mapped to $\gr^{(1)} \widehat{I_1}$ and then $\gr^{(2)} \widehat{I_1}$. Hence by the universal property of quotient we get that the following rectangle commutes:
		$$\begin{tikzcd}
			\Sym^{(2)}_{\bullet}(\g) \arrow[r, "\sim"] \arrow[d, two heads] & \gr^{(2)}_{\bullet} \Univ(\g) \arrow[d, two heads] \\
			\left(\Sym(\g)/I_1\right)^{(2)}_{\bullet} \arrow[r, "\sim"] & \gr^{(2)}_{\bullet} Q_1 
		\end{tikzcd} $$
	\end{claim}
	
	\begin{proof}[Proof of the claim]
		We consider a decomposition $ \g = \m_1 \oplus \p $ compatible with $\Gamma^2$, a basis $(y_1, \dots, y_m)$ a basis of $\m_1$ and $(y_{m+1}, \dots, y_d)$ a basis of $\p$ which consist of $q_1$ and $q_2$-homogeneous vectors: $y_{\ell} \in \g^{(1)}_{j_{\ell}} \cap \g^{(2)}_{k_{\ell}}$. 
		
		The ideal $I_1$ is spanned by the monomials
		$$ {y_1}^{i_1} \cdots {y_d}^{i_d} (y_q - \chi_1(y_q)), $$
		the ideal $\gr^{(1)}\widehat{I_1}$ is spanned by the monomials
		$$\left([y_1]^{(1)}_{2 - j_1}\right)^{i_1} \cdots \left([y_d]^{(1)}_{2 - j_d}\right)^{i_d}  \left([y_d]^{(1)}_{2 - j_q} - \chi_1(y_q)\right), $$
		and the ideal $\gr^{(2)}\widehat{I_1}$ is spanned by the monomials
		$$\left([y_1]^{(2)}_{2 - k_1}\right)^{i_1} \cdots \left([y_d]^{(2)}_{2 - k_d}\right)^{i_d}  \left([y_d]^{(2)}_{2 - k_q} - \chi_1(y_q)\right),$$
		where $1 \leqslant q \leqslant m$ and $i_1, \dots, i_{\ell} \in \Z_{\geqslant 0} $.
		
		The claim is now clearly true.
	\end{proof}

    Let us choose a nice PBW basis of $\Univ(\g, f_1)$ which is provided by \cite{premet2002slices}. We can consider a decomposition of $\g$ as $\C$-vector space $ \g = \m_1 \oplus \mathfrak{a} \oplus  \Ker \ad(f_1)$, each summand is chosen as $q_1$ and $q_2$-stable. Since $[q_1, q_2] = 0$ by assumption, we can set $(y_1, \dots, y_c)$ an ordered basis of $\Ker \ad(f)$ and $(y_{c+1}, \dots, y_p)$ an ordered basis of $\mathfrak{a}$ which are $q_1$, $q_2$ and $q_0$-homogeneous: \begin{equation} y_{\ell} \in \g^{(1)}_{j_{\ell}} \cap \g(q_0, \omega_{\ell}) \subseteq \g^{(2)}_{j_{\ell} + \omega_{\ell}} \label{equation:pbw} \end{equation}
	for $1 \leqslant \ell \leqslant p$.
	
    According to \cite[Theorem 4.6]{premet2002slices}, the W-algebra admits a PWB basis: this means that one can find elements of the form
	$$ Y_{u} \defeq y_u + \sum_{i_1, \dots, i_p \in \Z} \alpha_{i_1, \dots, i_p}^u {y_1}^{i_1} \cdots {y_p}^{i_p} \bmod \widehat{I_{1}},$$
	for $1 \leqslant u \leqslant c$, such that the ordered monomials 
	$$ {Y_1}^{i_1} \cdots {Y_c}^{i_c}, \quad \text{where} \quad i_1, \dots, i_c \in \Z_{\geqslant 0}, $$
	form a $\C$-linear basis of $\Univ(\g, f_1)$. 
		
	\begin{remark}
		Premet's original proof uses positive characteristic machinery. In \cite[Theorem 3.6]{brundan2008highest-weight}, one can find another proof using only characteristic 0.
	\end{remark}
 
	To carry on, we need to introduce the action of $q_0$ on $\Univ(\g, f_1)$.
	
	\begin{claim}
		The adjoint action of $q_0$ on $\Univ(\g)$ descends to $\Univ(\g, f_1)$ and the PBW basis (\ref{equation:pbw}) is $q_0$-homogeneous.
	\end{claim}
	
	\begin{proof}[Proof of the claim]
		Let us prove that $q_0 \bmod \widehat{I_1} \in \Univ(\g, f_1)$. We notice that for all $y \in \m_0$:
		$$ \chi_1([y, q_0]) = (f_1 \mid [y, q_0]) = (y \mid [q_0, f_1]) = 0 $$
		because $[q_0, f_1] = 0$. Hence, since $\m_1$ is $q_0$-stable:
		$$ [y, q_0] = [y, q_0] - \chi_1([y,q_0]) \in \widehat{I_{1}}. $$
		Hence this action is well-defined. Because $q_0$ is semisimple, this action is diagonalizable. 	
		
		We compute the bracket $[q_0, Y_u]$, it is equal to
		$$ \omega_u y_u + \sum_{i_1, \dots, i_p \in \Z} (i_1 \omega_1 + \cdots + i_p \omega_p) \alpha_{i_1, \dots, i_p}^u {y_1}^{i_1} \cdots {y_p}^{i_p} \bmod \widehat{I_{1}}. $$
		$[q_0, Y_u]$ can be written in a unique way in the PBW basis, which is entirely determined by the first term $\omega_u y_u$, which is the same as $\omega_u Y_u$. Whence $[q_0, Y_u] = \omega_u Y_u$. 
	\end{proof}
	
	As a side product of the previous proof, we get that when $\alpha_{i_1, \dots, i_p}^u$ is not zero, then necessarily $i_1 \omega_1 + \cdots + i_p \omega_p = \omega_u$. Hence we can write:
	$$ Y_{u} = y_u + \sum_{\substack{i_1, \dots, i_p \in \Z_{\geqslant 0} \\ i_1 \omega_1 + \cdots + i_p \omega_p = \omega_u}} \alpha_{i_1, \dots, i_p}^u {y_1}^{i_1} \cdots {y_p}^{i_p} \bmod \widehat{I_{1}} $$
	Moreover, if we add the condition $\sum_{\ell = 1}^p i_{\ell}(2 - j_{\ell})= 2 - j_u$, we have also:
	$$ \sum_{\ell = 1}^p i_{\ell}(2 - j_{\ell} - \omega_{\ell})= 2 - j_u - \omega_u.$$ 
	
	\begin{claim}
		The (non-graded) isomorphism
        $$ \gr^{(1)} Q_1 \overset{\sim}{\longrightarrow} \gr^{(2)} Q_1, $$
        induced by the Formula (\ref{equation:change-grading}), restricts to an isomorphism
        $$ \gr^{(1)} \Univ(\g, f_1) \cong \gr^{(2)} \Univ(\g, f_1). $$
	\end{claim}
	
	\begin{proof}[Proof of the claim]
		By restriction, (\ref{equation:change-grading}) induces an injection 
        $$ \gr^{(1)} \Univ(\g, f_1) \longrightarrow \gr^{(2)} \Univ(\g, f_1). $$
        To prove the claim, we need to prove surjectivity and to do so, we use the PBW basis: it is enough to prove that (\ref{equation:change-grading}) makes $[Y_u]^{(1)}_{1 - j_u}$ mapped to $[Y_u]^{(2)}_{1 - j_u - \omega_u}$. 
		
		To simplify the notations, we omit ``$\bmod \widehat{I_1}$'' in the definition of $Y_u$. We already know that
		$$ Y_u = y_u + \sum_{i_1, \dots, i_p \in \Z_{\geqslant 0}} \alpha_{i_1, \dots, i_p}^u {y_1}^{i_1} \cdots {y_p}^{i_p}. $$
		In the associated graded algebra, 
		$$ [Y_u]^{(1)}_{2 - j_u} = [y_u]^{(1)}_{2 - j_u} + \sum_{\sum_{\ell=1}^p i_{\ell}(2 - j_{\ell}) = 2 - j_u} \alpha_{i_1, \dots, i_p}^u \left[{y_1}^{i_1} \cdots {y_p}^{i_p}\right]^{(1)}_{2 - j_u}. $$
		We use the definition of the graded algebra product to write:
		$$ \left[{y_1}^{i_1} \cdots {y_p}^{i_p}\right]^{(1)}_{2 - j_u} = \left(\left[{y_1}\right]^{(1)}_{2 - j_1}\right)^{i_1} \cdots \left(\left[{y_p}\right]^{(1)}_{2 - j_p}\right)^{i_p}. $$
		After taking the sum, we get that $[Y_u]^{(1)}_{2 - j_u}$ is equal to
		$$  [y_u]^{(1)}_{2 - j_u} + \sum_{\sum_{\ell=1}^p i_{\ell}(2 - j_{\ell}) = 2 - j_u} \alpha_{i_1, \dots, i_p}^u \left(\left[{y_1}\right]^{(1)}_{2 - j_1}\right)^{i_1} \cdots \left(\left[{y_p}\right]^{(1)}_{2 - j_p}\right)^{i_p}. $$
		
		Under the map defined in (\ref{equation:change-grading}), we have:
		$$ \left[y_{\ell}\right]^{(1)}_{2 - j_{\ell}} \longmapsto \left[y_{\ell}\right]^{(2)}_{2 - j_{\ell} - \omega_{\ell}}. $$
		So, we deduce that $[Y_u]^{(1)}_{2 - j_u}$ is mapped to
		$$ [y_u]^{(2)}_{2 - j_u - \omega_u} + \sum_{\sum_{\ell=1}^p i_{\ell}(2 - j_{\ell}) = 2 - j_u} \alpha_{i_1, \dots, i_p}^u \left(\left[{y_1}\right]^{(2)}_{2 - j_1 - \omega_1}\right)^{i_1} \cdots \left(\left[{y_p}\right]^{(2)}_{2 - j_p - \omega_p}\right)^{i_p}.$$
		
		We use the consequence of the previous claim. In the sum defining $Y_u$, the terms are non-zero only if $ \sum_{\ell = 1}^p i_{\ell}(2 - j_{\ell} - \omega_{\ell})= 2 - j_u - \omega_u$. With this condition:
		$$  \left(\left[{y_1}\right]^{(2)}_{2 - j_1 - \omega_1}\right)^{i_1} \cdots \left(\left[{y_p}\right]^{(2)}_{2 - j_p - \omega_p}\right)^{i_p} = \left[{y_1}^{i_1} \cdots {y_p}^{i_p}\right]^{(2)}_{2 - j_u - \omega_u}, $$
		and taking the sum we get that $[Y_u]^{(1)}_{2 - j_u}$ is mapped to:
		$$ [y_u]^{(2)}_{2 - j_u - \omega_u} + \sum_{\sum_{\ell = 1}^p i_{\ell}(2 - j_{\ell} - \omega_{\ell})= 2 - j_u - \omega_u} \alpha_{i_1, \dots, i_p}^u \left[{y_1}^{i_1} \cdots {y_p}^{i_p}\right]^{(2)}_{2 - j_u - \omega_u} $$
		which is equal to $[Y_u]^{(2)}_{2 - j_u - \omega_u}$. 
		
		We can conclude that indeed we have the mapping
		$$ [Y_u]^{(1)}_{1 - j_u} \longmapsto [Y_u]^{(2)}_{1 - j_u - \omega_u}. $$
	\end{proof}
	
	By concatenating all the commutative diagrams we have built, we can conclude the proof of the proposition.
\end{proof}

\begin{remark}
	By similar argument as in the previous proof, one can show that the isomorphism $ \Poi^{(2)}_{\bullet} (\g, f_1) \cong \gr^{(2)}_{\bullet} \Univ(\g, f_1)$ is compatible with Poisson structures. Hence $\Univ(\g, f_1)$ is a quantization of $\Poi(\g, f_1)$ for the Kazhdan filtration induced by $\Gamma^2$. 
\end{remark}

The following proposition is an analogue of Theorem \ref{theorem:quantization-slodowy} for $\Univ_0$, the second stage of reduction. 

\begin{proposition} \label{proposition:green-diagram}
	The previous isomorphism $ \Poi^{(2)}_{\bullet}(\g, f_1)  \cong \gr^{(2)}_{\bullet} \Univ(\g, f_1) $ makes the following diagram commute.
	$$\begin{tikzcd}
		\Poi^{(2)}_{\bullet}(\g, f_1)  \arrow[r, "\sim"] \arrow[d, two heads] & \gr^{(2)}_{\bullet} \Univ(\g, f_1) \arrow[d, two heads] \\
		(\Poi(\g, f_1)/I_0)^{(2)}_{\bullet} \arrow[r, "\sim"] & \gr^{(2)}_{\bullet} Q_0 \\
		\left((\Poi(\g, f_1)/I_0)^{\ad(\m_0)}\right)^{(2)}_{\bullet} \arrow[r, "\sim"] \arrow[u, hook] & \gr^{(2)}_{\bullet} \Univ_0 \arrow[u, hook]
	\end{tikzcd} $$
\end{proposition}

\begin{proof}
	For this proof, we explain why the arguments of Gan and Ginzburg \cite{gan2002quantization} can be used in this context. As before we consider decompositions
	$$ \g = \m_1 \oplus \Ker \ad(f_1) \oplus \mathfrak{a}, \quad \Ker \ad(f_1) = \m_0 \oplus \mathfrak{b}, $$
	where all subspaces are equipped by gradings induced by $\Gamma_1$ and $\Gamma_2$. 
	
	Considering a $q_1$, $q_1$ and $q_0$-homogeneous basis adapted to this decomposition, we can take a PBW basis $(Y_u)_{1 \leqslant u \leqslant c}$ of $\Univ(\g, f_1)$ which contains a basis of $\m_0$. Hence, as we did in the proof of Claim \ref{claim:gr-ideals}, we can establish that the isomorphism $$ \Poi(\g, f_1) \overset{\sim}{\longrightarrow} \gr^{(2)} \Univ(\g, f_1) $$ maps $I_0$ to $\gr^{(2)} \widehat{I_0}$. Hence, the top rectangle commutes.
	
	In fact, the map $\Poi(\g, f_1)/I_0 \rightarrow \gr^{(2)} Q_0$ is $\m_0$-equivariant, hence for $m \in \Z_{\geqslant 0}$ we have the diagram:
	$$\begin{tikzcd}
		\left(\Poi(\g, f_1)/I_0\right)^{(2)}_m \arrow[r, "\sim"] & \gr^{(2)}_m Q_0 \\
		\left((\Poi(\g, f_1)/I_0)^{\ad(\m_0)}\right)^{(2)}_m \arrow[r, "\sim"] \arrow[u, hook] & \left(\gr^{(2)}_m Q_0 \right)^{\ad(\m_0)} \arrow[u, hook]
	\end{tikzcd}  $$
	To finish the proof of the claim, it remains to show that there is an isomorphism
	$$ \gr^{(2)}_m \Univ_0 \cong \left(\gr^{(2)}_m Q_0 \right)^{\ad(\m_0)}. $$ 
	
	To do so, as Gan and Ginzburg did, we introduce the \textit{Chevalley-Eilenberg cohomology} \cite[Section 7.7]{weibel1997} associated to $Q_0$. The corresponding complex is
	$$ Q_0 \otimes_{\C} \Alt^{\bullet}(\m_0)^*, $$
	where $\Alt^{\bullet}(\m_0)^*$ is the exterior algebra of $(\m_0)^*$. We recall that for any $\m_0$-module $V$, one has
	$$ \coH^{0}(V \otimes_{\C} \Alt^{\bullet}(\m_0)^*) = V^{\m_0}. $$
	
	If we rephrase the desired statement, we need to show an isomorphism
	$$ \gr^{(2)}_m \coH^0(Q_0 \otimes_{\C} \Alt^{\bullet}(\m_0)^*) \cong \coH^0\left(\gr^{(2)}_m \left(Q_0 \otimes_{\C} \Alt^{\bullet}(\m_0)^*\right)\right). $$
	To do so, we use standard arguments using spectral sequences: as a reference, we will use \cite[Chapter 5]{weibel1997}. 
	
	The complex $Q_0 \otimes_{\C} \Alt^{\bullet}(\m_0)^*$ is filtered as follows. The algebra $\m_0$ is $\Gamma^2$-graded with positive degrees, so we have an induced grading:
	$$ (\m_0)^* = \bigoplus_{k \leqslant -1} \left((\m_0)^*\right)^{(2)}_k. $$
	Hence we define $\left(Q_0\otimes_{\C} \Alt^{n}(\m_0)^*\right)^{(2)}_m$ as the subspace spanned by elements of the form  $X \otimes (\eta_1 \wedge \cdots \wedge \eta_n)  $ where $X \in (Q_0)^{(2)}_i $ and $\eta_1 \in \left((\m_0)^*\right)^{(2)}_{-k_1}, \dots, \eta_i \in \left((\m_0)^*\right)^{(2)}_{-k_i}$ verify $i + k_1 + \cdots + k_n \leqslant m$.
	
	We get a filtered complex and by \cite[Section 5.4]{weibel1997} there is an associated spectral sequence $(E_r)_{r \geqslant 0}$ such that, for all $p, q \in \Z $:
	\begin{align*}
		E^{m,n}_0 & = \gr^{(2)}_m \left(Q_0\otimes_{\C} \Alt^{m + n}(\m_0)^*\right), \\
		E^{m,n}_1 & = \coH^{m+n}\left(\gr^{(2)}_m \left(Q_0\otimes_{\C}  \Alt^{\bullet}(\m_0)^*\right)\right)  \cong \coH^{m+n}\left(\left(\C[{\mu_0}^{-1}(0)] \otimes_{\C} \Alt^{\bullet}(\m_0)^*\right)^{(2)}_m\right).
	\end{align*}
	where the second isomorphism is a consequence of the $\m_0$-isomorphism:
	$$ \gr^{(2)} Q_0  \cong \C[{\mu_0}^{-1}(0)]. $$
	
	\begin{claim}
		The filtration on $\Univ(\g, f_1) / \widehat{I_0}$ is non-negative.
	\end{claim}
	
	\begin{proof}[Proof of the claim]
		In the decomposition $ \g =  \m_1 \oplus \m_0 \oplus \mathfrak{b} \oplus \mathfrak{a},$ the subspace $\m_2 = \m_1 \oplus \m_0$ corresponds to the negative generators for the filtration $\Univ^{(2)}_{\bullet}(\g)$ and $\mathfrak{b} \oplus \mathfrak{a}$ to the positive generators. Hence, in $ \Ker \ad(f_1) = \m_0 \oplus \mathfrak{b}, $ the subspace $\m_0$ corresponds to negative generators of $\Univ(\g, f_1)$ and $\mathfrak{b}$ to positive ones. So, after taking the quotient by $\widehat{I_0}$, only positive generators provided by $\mathfrak{b}$ remain in $Q_0$.
	\end{proof}
	
	As a direct consequence of the previous claim, $\left(Q_0 \otimes_{\C} \Alt^{\bullet}(\m_0)^*\right)^{(2)}_{\bullet}$ is non-negatively filtered, the filtration is exhaustive and so, by \cite[Theorem 5.5.1]{weibel1997}, the spectral sequence converges to
    $$ E^{m, n}_{\infty} \cong \gr^{(2)}_{m} \coH^{m + n}\left(Q_0 \otimes_{\C} \Alt^{\bullet}(\m_0)^*\right), \quad p, q \in \Z. $$
	
	\begin{claim} \label{claim:derham}
		We have $\coH^{n}\left(\C[{\mu_0}^{-1}(0)] \otimes_{\C} \Alt^{\bullet}(\m_0)^*\right) = \{0\}$ for $n \neq 0$. 
	\end{claim}
	
	\begin{proof}[Proof of the claim]
		According to Theorem \ref{theorem:freeness}, as $\m_0$-modules we have the isomorphism:
		$$ \C[{\mu_0}^{-1}(0)] \cong \C[M_0] \otimes_{\C} \C[\Slice_2]. $$
		Hence:
		\begin{align*}
			\coH^{n}\left(\C[{\mu_0}^{-1}(0)] \otimes_{\C} \Alt^{\bullet}(\m_0)^*\right) & \cong \coH^{n}\left(\C[M_0] \otimes_{\C} \C[\Slice_2] \otimes_{\C} \Alt^{\bullet}(\m_0)^*\right) \\
			& \cong \coH^{n}\left(\C[M_0] \otimes_{\C} \Alt^{\bullet}(\m_0)^*\right) \otimes_{\C} \C[\Slice_2].
		\end{align*}
		But $\coH^{\bullet}\left(\C[M_0] \otimes_{\C} \Alt^{\bullet}(\m_0)^*\right)$ corresponds in fact to the algebraic de Rham cohomology of $M_0$. The group $M_0$ is unipotent, so as an algebraic variety it is isomorphic to the affine space $\C^{\dim M_0}$. Its de Rham cohomology is trivial except at degree 0, where the corresponding cohomology group is isomorphic to $\C$. Whence
		$$ \coH^{\bullet}\left(\C[M_0] \otimes_{\C} \Alt^{\bullet}(\m_0)^*\right) = \coH^{0}\left(\C[M_0] \otimes_{\C} \Alt^{\bullet}(\m_0)^*\right) \cong \C. $$
		The claim follows.
	\end{proof}
	
	The previous claim implies that the spectral sequence collapses: the sequence $E_{r}$ is constant for $r \geqslant 1$ because $E_1$ is supported on one diagonal line. Hence, the natural map
	$$ \gr^{(2)}_m\coH^0\left(Q_0 \otimes_{\C} \Alt^{\bullet}(\m_0)^*\right) \longrightarrow \coH^0\left(\gr^{(2)}_m \left(Q_0\otimes_{\C}  \Alt^{\bullet}(\m_0)^*\right)\right) $$
	is an isomorphism and we get the second rectangle of the diagram, which concludes the proof of the proposition.
\end{proof}

Finally, the last proposition is a refinement of Theorem \ref{theorem:quantization-slodowy}.

\begin{proposition} \label{proposition:purple-diagram}
	The isomorphism $ \left(\Sym(\g)/I_2\right)^{(2)}_{\bullet} \cong \gr^{(2)}_{\bullet} Q_2 $ provided by Theorem \ref{theorem:quantization-slodowy} make the following diagram commute
	$$\begin{tikzcd}
		\left(\Sym(\g)/I_2\right)^{(2)}_{\bullet}\arrow[r, "\sim"] & \gr^{(2)}_{\bullet} Q_2  \\
		\left(\left(\Sym(\g)/I_2\right)^{\ad(\m_1)}\right)^{(2)}_{\bullet}  \arrow[r, "\sim"] \arrow[u, hook]& \gr^{(2)}_{\bullet} {Q_2}^{\ad(\m_1)} \arrow[u, hook] \\
		\left(\left(\Sym(\g)/I_2\right)^{\ad(\m_2)}\right)^{(2)}_{\bullet} \arrow[r, "\sim"] \arrow[u, hook] & \gr^{(2)}_{\bullet} \Univ(\g, f_2) \arrow[u, hook]
	\end{tikzcd} $$
	is commutative.
\end{proposition}

\begin{proof}[Sketch of proof]
	The idea of the proof rely on spectral sequences, as before. For the first rectangle, we introduce the spectral sequence
	\begin{align*}
		E^{m,n}_0 & = \gr^{(2)}_m \left(Q_2\otimes_{\C} \Alt^{m + n}(\m_1)^*\right), \\
		E^{m,n}_1 & = \coH^{m+n}\left(\gr^{(2)}_m \left(Q_2\otimes_{\C}  \Alt^{\bullet}(\m_1)^*\right)\right) \cong \coH^{m+n}\left(\left(\C[{\mu_2}^{-1}(0)] \otimes_{\C} \Alt^{\bullet}(\m_1)^*\right)^{(2)}_m\right)..
	\end{align*}
	The filtration is non-negative and exhaustive, so the spectral sequence converges to
    $$ E^{m,n}_{\infty} \cong \gr^{(2)}_{m} \coH^{m + n}\left(Q_2 \otimes_{\C} \Alt^{\bullet}(\m_1)^*\right) $$
	and it collapses because one has $ \C[{\mu_2}^{-1}(0)] \cong \C[M_1] \otimes \C[\widetilde{\mu_0}^{-1}(0)] $
	according to Claim \ref{claim:isomorphism}.
	
	For the second rectangle, we introduce the spectral sequence
	\begin{align*}
		E^{m,n}_0 & = \gr^{(2)}_m \left({Q_2}^{\ad(\m_1)}\otimes_{\C} \Alt^{m + n}(\m_0)^*\right), \\
		E^{m,n}_1 & = \coH^{m+n}\left(\gr^{(2)}_m \left({Q_2}^{\ad(\m_1)}\otimes_{\C}  \Alt^{\bullet}(\m_0)^*\right)\right) \\
		& \cong \coH^{m+n}\left(\left(\C[{\mu_0}^{-1}(0)] \otimes_{\C} \Alt^{\bullet}(\m_0)^*\right)^{(2)}_m\right).
	\end{align*}
	The filtration is non-negative so it converges to
    $$ E^{m, n}_{\infty} \cong \gr^{(2)}_{m} \coH^{m+n}\left({Q_2}^{\ad(\m_1)} \otimes_{\C} \Alt^{\bullet}(\m_0)^*\right) $$ and it collapses because one has $ \C[{\mu_0}^{-1}(0)] \cong \C[M_0] \otimes \C[\Slice_2] $ according to Claim~\ref{claim:freeness}.
\end{proof}

The two previous propositions provide the isomorphisms
$$ \left((\Poi(\g, f_1) / I_0)^{\ad(\m_0)}\right)^{(2)}_{\bullet} \cong \gr^{(2)}_{\bullet} \Univ_0 \quad \text{and} \quad \Poi^{(2)}_{\bullet}(\g, f_2)  \cong \gr^{(2)}_{\bullet} \Univ(\g, f_2) $$
we expected for Theorem \ref{theorem:gr-two-stages}. Its proof is then achieved.

\section{Hook-type nilpotent orbits and other examples}
\label{section:hook}

In this section, we provide several examples in classical type for the Hamiltonian reduction by stage. It is well-known that in all classical types, the nilpotent orbits of a simple Lie algebra can be classified using partitions \cite[Chapter 5]{collingwood1993nilpotent}. We denote a \textit{partition} $\lambda$ of a positive integer $n$ by a sequence of positive integers $\lambda \defeq (\lambda_1, \dots, \lambda_r)$ such that
$$ \lambda_1 \geqslant \cdots \geqslant \lambda_r \geqslant 1, \quad \lambda_1 + \cdots + \lambda_r = n. $$
When in a subsequence $(\lambda_a, \dots, \lambda_b)$ of $\lambda$, all $\lambda_i$'s are equal to an integer $c$, the subsequence is compactly written as $(c^{b-a+1})$. A partition $\lambda = (\lambda_1, \dots, \lambda_r)$ of $n$ can be represented by its associated \textit{Young diagram}.

To classify good gradings in classical types, Elashvili and Kac introduced the notion of \textit{pyramid} in \cite{elashvili2005grading}. Roughly speaking, pyramids are constructed by moving blocks in Young diagrams. To draw our pyramids, we use the same conventions as in \cite{brundan2007good}. From this pyramid, one can build a nilpotent element $f$ and grading $\Gamma$ defined by the adjoint action of a semisimple element $q$, which is good for $f$. To do it, we follow the convention of \cite{arakawa2022singularities}, which is slightly different from \cite{brundan2007good}. In \cite{brundan2007good}, they fixe $e$ and they deal with good gradings $\Gamma$ such that $e \in \g_2$. In \cite{arakawa2022singularities} and the present paper, one takes an alternative convention by fixing $f$ and considering $\Gamma$ such that $f \in \g_{-2}$.

We denote by $e_{i,j}$, for $1 \leqslant i,j \leqslant n$, the canonical basis of $\Mat_n$, the vector space of square matrices.

\subsection{Examples in type A} Let $r \geqslant 1$ be an integer, set $n \defeq r + 1$ and let $\sl_{n}$ be the Lie algebra of traceless matrices of size $n$, which is the simple Lie algebra of type $\mathrm{A}_r$.

\subsubsection{Hook-type partitions} 

We consider a hook-type partition of $n$: $$ \lambda^{(\ell)} \defeq (\ell, 1^{n - \ell}), \quad 1 \leqslant \ell \leqslant n. $$ The following left-aligned pyramid of shape $\lambda^{(\ell)}$ represented in Figure \ref{figure:hook}, page \pageref{figure:hook}, defines an even good grading for $ f^{(\ell)} \defeq e_{2,1} + e_{3,2} + \cdots + e_{\ell, \ell - 1}. $ The corresponding good algebra is given by
$$ \m^{(\ell)}  \defeq \operatorname{Span}\{e_{i,j} ~ | ~ 1 \leqslant i \leqslant \ell - 1, ~ i < j \leqslant n\}. $$

\begin{proposition} \label{proposition:hook}
	For any integers $1 \leqslant \ell_1 < \ell_2 \leqslant n$, set $f_1 \defeq f^{(\ell_1)}$ and $f_2 \defeq f^{(\ell_2)}$. Then the Slodowy slice of $f_2$ is a Hamiltonian reduction of the slice of $f_1$. Moreover, the finite W-algebra of $f_2$ is a Hamiltonian reduction of the W-algebra of $f_1$. 
\end{proposition}

\begin{proof}
	We apply Theorems \ref{theorem:geometrical-stage-reduction} and \ref{theorem:quantum-stage-reduction} to $\g = \sl_{n}$, $f_1$, $f_2$ and $\m_i \defeq \m^{(\ell_i)}$.
\end{proof}

\subsubsection{Rectangular and minimal nilpotent orbits in rank $3$}

Set $n = 4$. Let us pick
$$ f_1 \defeq e_{2,1} \quad \text{and} \quad f_2 \defeq  e_{2,1} + e_{4,3} $$
two nilpotent elements in $\sl_4$, respectively associated to the partitions $(2,1^2)$ and $(2^2)$. Their are obviously embedded in $\sl_2$-triples $(e_i, h_i, f_i)$, $i=1,2$, where 
$$e_1 \defeq e_{1,2}, \quad e_2 \defeq  e_{1,2} + e_{3,4} $$
and
$$ h_1 \defeq e_{1,1} - e_{2,2}, \quad h_2 \defeq e_{1,1} - e_{2,2} + e_{3,3} - e_{4,4}. $$

The adjoint actions of $h_1$ and $h_2$
induce good gradings (in fact Dynkin gradings) for, respectively, $f_1$ and $f_2$. The grading induced by $h_2$ is even and the corresponding unipotent group is
$$ M_2 = \left\{\left(\begin{smallmatrix}
	1 & * & 0 & * \\
	0 & 1 & 0 & 0 \\
	0 & * & 1 & * \\
	0 & 0 & 0& 1
\end{smallmatrix}\right)\right\}. $$
With the good choice of a Lagrangian subspace, a unipotent group corresponding to $f_1$ is
$$ M_1 = \left\{\left(\begin{smallmatrix}
	1 & * & 0 & * \\
	0 & 1 & 0 & 0 \\
	0 & * & 1 & 0  \\
	0 & 0 & 0 & 1
\end{smallmatrix}\right)\right\}. $$

\begin{proposition} 
	The Slodowy slice of $f_2$ is a Hamiltonian reduction of the slice of $f_1$. In the same way, the finite W-algebra of $f_2$ is a Hamiltonian reduction of the W-algebra of $f_1$. 
\end{proposition}

\begin{proof}
	We apply Theorems \ref{theorem:geometrical-stage-reduction} and \ref{theorem:quantum-stage-reduction} to $\g = \sl_{4}$, $f_1$, $f_2$, $M_1$ and $M_2$.
\end{proof}

\subsection{Examples in type B} Set $n \defeq 2r + 1$ for $r \geqslant 2$ an integer. We realize the simple Lie algebra of type $\mathrm{B}_r$ as the following subalgebra of $\sl_n$:
$$ \so_n \defeq \{x \in \sl_n ~ | ~ x^{\mathsf{T}} K + K x= 0 \}, \quad K \defeq \left(\begin{smallmatrix}
	(0) & & 1 \\
	& \iddots &  \\
	1 & & (0)
\end{smallmatrix}\right). $$
where $x^{\mathsf{T}}$ denotes the transpose of $x$. As in \cite{elashvili2005grading, brundan2007good}, we take the symmetric set
$$ I_n \defeq \{-r, \dots, -2, -1, 0,  1, 2, \dots, r\} $$
as indexation for the canonical basis of $\C^n$, with the following order:
$$ v_{-r} = (1,0, \dots, 0), \quad v_{-r + 1} = (0, 1, \dots, 0), \quad \dots \quad v_{r} = (0, 0, \dots, 1). $$
We change the numbering of the elementary matrices $e_{i,j}$ to have $i,j \in I_n$ and to respect the order we have chosen for the basis.

Consider the regular nilpotent element
$$ f_2 \defeq  \sum_{i=0}^{r-1} e_{i+1, i} - \sum_{i=0}^{r-1} e_{-i, -i-1} $$
corresponding to the regular orthogonal partition $\lambda_{2} \defeq (n)$. We consider the Dynkin grading which is induced by the adjoint action of the diagonal matrix $ q_2 \defeq \operatorname{diag}(2r ,\dots, 4, 2,0, -2, -4, \dots ,-2r). $ The corresponding good algebra is
the nilradical of $\so_n$ consisting in strictly upper-triangular matrices: $$\m_2 = \Span\left\{e_{i,j}~  | ~   i < j \in I_n \right\} \cap \so_n. $$

We consider the subregular orthogonal partition $\lambda_1 \defeq (n - 2, 1^2)$ of $n$. We associate to it the orthogonal pyramid of Figure \ref{figure:orth-subregular}. From this pyramid we can construct the nilpotent element $$f_1 \defeq \sum_{i=0}^{r-2} e_{i+1, i} - \sum_{i=0}^{r-2} e_{-i, -i-1} $$ of type $\lambda_1$ and an even good grading for $f_1$ induced by the adjoint action of $ q_1 \defeq \operatorname{diag}(2r - 2, 2r -2, 2r -4, 2r- 6, \dots, 2, 0, 2, \dots, 6 - 2r, 4 - 2r, 2 - 2r, 2 - 2r)$. The corresponding good algebra is:
$$ \m_1 = \Span\left\{e_{i,j} ~  | ~ i < j \in I_n, ~ i < r - 1, ~  j > - r + 1 \right\} \cap \so_n. $$

By applying Theorems \ref{theorem:geometrical-stage-reduction} and \ref{theorem:quantum-stage-reduction} to $f_1$, $f_2$, $\m_1$ and $\m_2$, we get the following proposition.

\begin{proposition}
	The Slodowy slice to a regular nilpotent element in $\so_n$ is a Hamiltonian reduction of the slice to a subregular nilpotent element in $\so_n$. Moreover, the regular finite W-algebra of $\so_n$ is a Hamiltonian reduction of the finite W-algebra associated to a subregular nilpotent element in $\so_n$. 
\end{proposition}

\subsection{Examples in type C} 

Set $n = 2r$ for $r \geqslant 3$ an integer. We realize the simple Lie algebra of type $\mathrm{C}_r$ as the following subalgebra of $\sl_n$:
$$ \sp_n \defeq \{x \in \sl_n ~ | ~ x^{\mathsf{T}} J + J x= 0 \}, \quad J \defeq \left(\begin{smallmatrix}
	& & & & & 1 \\
	& (0) & & & \iddots & \\
	& & & 1 & & \\
	& & -1 & & & \\
	& \iddots & & & (0) & \\
	-1 & & & & &
\end{smallmatrix}\right). $$
As in \cite{elashvili2005grading, brundan2007good}, we take the symmetric set
$$ I_n \defeq \{-r, \dots, -2, -1, 1, 2, \dots, r\} $$
as indexation for the canonical basis of $\C^n$, with the following order:
$$ v_{-r} = (1,0, \dots, 0), \quad v_{-r + 1} = (0, 1, \dots, 0), \quad \dots \quad v_{r} = (0, 0, \dots, 1). $$
We change the numbering of the elementary matrices $e_{i,j}$ to have $i,j \in I_n$ and to respect the order we have chosen for the basis.

Consider the regular nilpotent element
$$ f_2 \defeq e_{1,-1} + \sum_{i=1}^{r-1} e_{i+1, i} - \sum_{i=1}^{r-1} e_{-i, -i-1} $$
corresponding to the partition $\lambda_{2} \defeq (n)$. We consider the Dynkin grading induced by the adjoint action of $ q_2 \defeq \operatorname{diag}(2r - 1,\dots, 3, 1,-1, -3, \dots ,1 -2r). $ The corresponding good algebra is
the nilradical of $\sp_n$ consisting in strictly upper-triangular matrices: $$\m_2 = \Span\left\{e_{i,j}~  | ~   i < j \in I_n \right\} \cap \sp_n. $$

We consider the symplectic partition $\lambda_1 \defeq (2^2,1^{n-4})$ of $n$. It corresponds to the second smallest nonzero nilpotent orbit, after the minimal one. We associate to it the symplectic pyramid of Figure \ref{figure:overminimal}. From this pyramid we can construct the nilpotent element $$f_1 \defeq e_{r,r-1} - e_{-r +1,-r}$$ of type $\lambda_1$ and an even good grading for $f_1$ induced by the adjoint action of $ q_1 \defeq \operatorname{diag}(2,0,\dots, 0, 0, \dots ,0 ,-2)$. The corresponding good algebra is:
$$ \m_1 = \Span\left\{e_{-r, i}, ~  e_{j,r} ~  | ~  i > -r, ~  j < r \right\} \cap \sp_n. $$

By applying Theorems \ref{theorem:geometrical-stage-reduction} and \ref{theorem:quantum-stage-reduction} to $f_1$, $f_2$, $\m_1$ and $\m_2$, we get the following proposition.

\begin{proposition} 
	In type $\mathrm{C}_r$, $r \geqslant 3$, the regular Slodowy slice is a Hamiltonian reduction of the slice corresponding to $(2^2,1^{n-4})$. In the same way, the regular finite W-algebra is a Hamiltonian reduction of the the slice corresponding to $(2^2,1^{n-4})$. 
\end{proposition}

\begin{figure}[t]
	\begin{minipage}[b]{0.45\textwidth}
		\centering
		$$\begin{picture}(100,80)
			\put(0,0){\line(1,0){45}}
			\put(55,0){\line(1,0){45}}
			\put(0,20){\line(1,0){45}}
			\put(55,20){\line(1,0){45}}
			\put(0,40){\line(1,0){20}}
			\put(0,60){\line(1,0){20}}
			\put(0,80){\line(1,0){20}}
			\put(0,0){\line(0,1){45}}
			\put(20,0){\line(0,1){45}}
			\put(0,55){\line(0,1){25}}
			\put(20,55){\line(0,1){25}}
			\put(40,0){\line(0,1){20}}
			\put(60,0){\line(0,1){20}}
			\put(80,0){\line(0,1){20}}
			\put(100,0){\line(0,1){20}}
			\put(90,10){\makebox(0,0){{$1$}}}
			\put(70,10){\makebox(0,0){{$2$}}}
			\put(51,10){\makebox(0,0){{$\cdots$}}}
			\put(30,10){\makebox(0,0){{\tiny $\ell - 1$}}}
			\put(10,10){\makebox(0,0){{$\ell$}}}
			\put(10,30){\makebox(0,0){{\tiny $\ell + 1$}}}
			\put(10,53){\makebox(0,0){{$\vdots$}}}
			\put(10,70){\makebox(0,0){{$n$}}}
		\end{picture}$$
		\caption{\small Pyramid for hook partition $(\ell,1 , \dots, 1)$}
		\label{figure:hook}
		
		$$\begin{picture}(100,60)
			\put(0,0){\line(1,0){20}}
			\put(0,20){\line(1,0){45}}
			\put(55,20){\line(1,0){45}}
			\put(0,40){\line(1,0){45}}
			\put(55,40){\line(1,0){45}}
			\put(80,60){\line(1,0){20}}
			\put(0,0){\line(0,1){40}}
			\put(20,0){\line(0,1){40}}
			\put(40,20){\line(0,1){20}}
			\put(60,20){\line(0,1){20}}
			\put(80,20){\line(0,1){40}}
			\put(100,20){\line(0,1){40}}
			\put(10,10){\makebox(0,0){$r$}}
			\put(10,30){\makebox(0,0){{\tiny $r-1$}}}
			\put(30,30){\makebox(0,0){{\tiny $r-2$}}}
			\put(51,30){\makebox(0,0){$\cdots$}}
			\put(70,30){\makebox(0,0){{\tiny $2-r$}}}
			\put(90,30){\makebox(0,0){{\tiny $1-r$}}}
			\put(90,50){\makebox(0,0){$-r$}}
		\end{picture}$$
		\caption{\small Orthogonal pyramid for subregular parition $(n - 2,1, 1)$}
		\label{figure:orth-subregular}
	\end{minipage}%
\hspace{0.1\textwidth}
	\begin{minipage}[b]{0.45\textwidth}
		\centering
		\begin{picture}(60,120)
			\put(20,0){\line(1,0){20}}
			\put(20,20){\line(1,0){20}}
			\put(0,40){\line(1,0){40}}
			\put(0,60){\line(1,0){60}}
			\put(20,80){\line(1,0){40}}
			\put(20,100){\line(1,0){20}}
			\put(20,120){\line(1,0){20}}
			
			\put(20,0){\line(0,1){25}}
			\put(40,0){\line(0,1){25}}
			\put(0,40){\line(0,1){20}}
			\put(20,35){\line(0,1){50}}
			\put(40,35){\line(0,1){50}}
			\put(60,60){\line(0,1){20}}
			\put(20,95){\line(0,1){25}}
			\put(40,95){\line(0,1){25}}
			
			\put(30,10){\makebox(0,0){$1$}}
			\put(30,50){\makebox(0,0){\tiny $r-1$}}
			\put(10,50){\makebox(0,0){$r$}}
			\put(30,33){\makebox(0,0){$\vdots$}}
			
			\put(30,93){\makebox(0,0){$\vdots$}}
			\put(50,70){\makebox(0,0){$-r$}}
			\put(30,70){\makebox(0,0){\tiny $1-r$}}
			\put(30,110){\makebox(0,0){$-1$}}
		\end{picture}
		\caption{\small Symplectic pyramid for partition $(2,2,1, \dots, 1)$}
		\label{figure:overminimal}
	\end{minipage}   
\end{figure}

\subsection{Examples in type G}
Let $\g$ be the simple Lie algebra of type $\mathrm{G}_2$. For a concrete construction of this exceptionnal Lie algebra, one can see \cite[Section 19.3]{humphreys1972book}, where $\g$ is build as a subalgebra of the classical simple algebra of type $\mathrm{B}_3$. Let $\h$ be a Cartan subalgebra of $\g$ and $\Delta_+$ be a set of positive roots of $\mathfrak{g}$ with simple roots $\Pi = \{\alpha_1, \alpha_2\}$, where $\alpha_1$ is a short root and $\alpha_2$ is a long root. Then $$\Delta_+ = \{\alpha_1, \alpha_2, \alpha_1+\alpha_2, 2\alpha_1+\alpha_2, 3\alpha_1+\alpha_2, 3\alpha_1+2\alpha_2\}.$$ Choose positive root vectors $e_\alpha$, for $\alpha \in \Delta_+$, and negative root vectors $f_\alpha$, for $\alpha \in \Delta_- = - \Delta_+$. Using the normalized symmetric invariant bilinear form on $\g$,  we have an isomorphism $\h \cong \h^*$. Denote by $h_i$ the vector in  $\h$ corresponding to $\alpha_i$ so that $\alpha_i(h_j) = (\alpha_i \mid \alpha_j)$. Let
\begin{align*}
	f_1 \defeq f_{2\alpha_1+\alpha_2},\quad
	f_2 \defeq f_{\alpha_2} + f_{2\alpha_1+\alpha_2}.
\end{align*}
Then $f_1$ is a nilpotent element labeled by $\widetilde{\mathrm A}_1$ in the Bala-Carter theory and $f_2$ is a subregular nilpotent element. We have $\mathfrak{sl}_2$-triples $(e_k, x_k, f_k)$ in $\mathfrak{g}$, for $k=1,2$, where
\begin{align*}
	x_1 \defeq 6h_1+3h_2,\quad
	x_2 \defeq 6h_1+4h_2.
\end{align*}
Equivalently, $x_1 = \varpi_1^\vee$ and $x_2 = 2\varpi_2^\vee$, where $\varpi_i^\vee$ is the $i$-the fundamental coweight in $\mathfrak{h}$ defined by $\alpha_i(\varpi^\vee_j) = \delta_{i, j}$. Let $\mathfrak{g} = \bigoplus_{j \in \Z}\mathfrak{g}_j^{(k)}$ be the Dynkin gradings on $\mathfrak{g}$ defined by $\operatorname{ad}(x_k)$ for $k=1,2$. Then
\begin{align*}
	\mathfrak{g}
	&= \mathfrak{g}^{(1)}_{-3} \oplus \mathfrak{g}^{(1)}_{-2} \oplus \mathfrak{g}^{(1)}_{-1} \oplus \mathfrak{g}^{(1)}_0 \oplus  \mathfrak{g}^{(1)}_{1} \oplus \mathfrak{g}^{(1)}_{2} \oplus \mathfrak{g}^{(1)}_{3}\\
	&= \mathfrak{g}^{(2)}_{-4} \oplus \mathfrak{g}^{(2)}_{-2} \oplus \mathfrak{g}^{(2)}_0 \oplus  \mathfrak{g}^{(2)}_{2} \oplus \mathfrak{g}^{(2)}_{4}
\end{align*}
and
\begin{align*}
	&\g^{(1)}_1 = \C e_{\alpha_1} \oplus \C e_{\alpha_1+\alpha_2},\quad
	\g^{(1)}_2 = \C e_{2\alpha_1+\alpha_2},\quad
	\g^{(1)}_3 = \C e_{3\alpha_1+\alpha_2} \oplus \C e_{3\alpha_1+2\alpha_2},\\
	&\g^{(2)}_2 = \C e_{\alpha_2} \oplus \C e_{\alpha_1+\alpha_2} \oplus \C e_{2\alpha_1+\alpha_2} \oplus \C e_{3\alpha_1+\alpha_2},\quad
	\g^{(2)}_4 = \C e_{3\alpha_1+2\alpha_2}.
\end{align*}
We may take $\mathfrak{l}_1 = \C e_{\alpha_1+\alpha_2}$ as a Lagrangian subspace in $\g^{(1)}_1$ with respect to the symplectic form $(x, y) \mapsto (f_1\mid[x, y])$. Set
\begin{align*}
	\m_1 = \mathfrak{l}_1\oplus \mathfrak{g}^{(1)}_2 \oplus \mathfrak{g}^{(1)}_3,\quad
	\m_2 = \mathfrak{g}^{(2)}_{2} \oplus \mathfrak{g}^{(2)}_{4}.
\end{align*}
Then $f_1, f_2, \m_1, \m_2$ satisfy the assumptions in Theorems \ref{theorem:geometrical-stage-reduction} and Theorem \ref{theorem:quantum-stage-reduction}.

\begin{proposition}
	Let $\g$ be the simple Lie algebra of type $\mathrm{G}_2$. Then the Slodowy slice to a subregular nilpotent element in $\g$ is a Hamiltonian reduction of the slice to a nilpotent element in $\g$ whose Bala-Carter label is $\widetilde{\mathrm A}_1$. Moreover, the subregular finite W-algebra of $\g$ is a Hamiltonian reduction of the finite W-algebra associated to a nilpotent element in $\g$ whose Bala-Carter label is $\widetilde{\mathrm A}_1$.
\end{proposition}

\section{The Skryabin equivalence}
\label{section:skryabin}

As Gan and Ginzburg did in \cite[Section 6]{gan2002quantization}, we want to apply our results to the study of the Skryabin equivalence of categories \cite{skryabin2002equivalence}. In particular, we state and prove Main Theorem \ref{main:skryabin-stage}.

\subsection{Original statement} 

Let $\chi$ be the linear form of $\g$ associated to the nilpotent element $f \in \g$. When we restrict it to the good algebra $\m$, $\chi|_{\m}$ is a character. A left $\Univ(\g)$-module is called a \textit{Whittaker module} with respect to $f$ and $\m$ if for all $y \in \m$, the element $y - \chi(y)$ acts locally nilpotently on $V$: that is to say, for all $v \in V$, there is some positive integer $n$ large enough such that: $ (y - \chi(y))^n \cdot v = 0. $ We denote by $\whcat(\g, f, \m)$ the category of Whittaker modules associated to $f$ which are finitely generated.

For any module $V \in \whcat(\g, f, \m)$, set:
$$ \Whi(V) \defeq \{v \in V ~ | ~ \forall y \in \m, ~ y \cdot v = \chi(y) v \}. $$
The set $\Whi(V)$ has a natural structure of $\Univ(\g, f)$-module. We also denote by $\module{\Univ(\g, f)}$  the category of finitely generated left $\Univ(\g, f)$-modules. 

Conversely, if we take $W \in \module{\Univ(\g, f)}$, we can define the tensor product 
$$\Ind(W) \defeq Q_f \otimes_{\Univ(\g, f)} W. $$ 
It is a $\Univ(\g)$-module and $\Ind(W) \in \whcat(\g, f, \m)$.

The following statement is the Skryabin equivalence.

\begin{theorem}[{\cite[Theorem 6.1]{gan2002quantization}}] \label{theorem:skryabin}
	We have an equivalence of categories:
	$$ \begin{tikzcd}[column sep = tiny]
		\whcat(\g, f, \m) \arrow[rr, bend left, "\Whi"] & \simeq & \module{\Univ(\g, f)} \arrow[ll, bend left, "\Ind"]
	\end{tikzcd} $$
\end{theorem}

\subsection{Category equivalences for the reduction by stages} 

From now, we fix $f_1, f_2$ two nilpotent elements of $\g$ and we adopt the same notations as in Subsection~ \ref{subsec:quantum-stage}. We assume the hypotheses of Theorem \ref{theorem:freeness}, in particular we have the following relation for the Lie algebras algebras:
$$ \m_2 = \m_1 \oplus \m_0, \quad [\m_1, \m_0] \subseteq \m_1. $$
The linear forms associated to $f_i$, $i=1,2$, denoted by $\chi_i$, coincide on $\m_1$. Recall the notations
$$ Q_0 = \Univ(\g, f_1) / \widehat{I_0} \quad \text{and} \quad \Univ_0 = Q_0{}^{\ad(\m_0)}. $$
We also established an algebra isomorphism: $$ \Univ_0 \cong \Univ(\g, f_2). $$

For $i=1,2$, we use the following notations for the Skryabin equivalence corresponding to $f_i$:
$$ \begin{tikzcd}[column sep = tiny]
	\whcat(\g, f_i, \m) \arrow[rr, bend left, "\Whi_i"] & \simeq & \module{\Univ(\g, f_i)} \arrow[ll, bend left, "\Ind_i"]
\end{tikzcd} $$

Recall that, according to Lemma \ref{lemma:m0-embedding}, $\m_0$ is embedded in $\Univ(\g, f_1)$; so the following definition makes sense. Let us denote by $\whcat_0$ the category of finitely generated left $\Univ(\g, f_1)$-modules $V$ such that for all $y \in \m_0$, the element $y - \chi_2(y)$ acts locally nilpotently on $V$: that is to say, for all $v \in V$ there is some positive integer $n$ large enough such that: $ (y - \chi_2(y))^n \cdot v = 0. $ For $V \in \whcat_0$, we set:
$$ \Whi_0(V) \defeq \{v \in V ~ | ~ \forall y \in \m_0, ~ y \cdot v = \chi_2(y) v \}. $$

\begin{lemma}
	For any $V \in \whcat_0$, $\Whi_0(V)$ is naturally a left $\Univ_0$-module, that is to say a $\Univ(\g, f_2)$-module.
\end{lemma}

\begin{proof}
    Set $v \in \Whi_0(V)$. By definition, for all $y \in \m_0$, $v$ is killed by the left generators of $\widehat{I_0}$:
    $(y - \chi_2(y)) \cdot v = 0$. Hence, we get that $ \widehat{I_0} \cdot v = 0. $ Therefore, $X \bmod \widehat{I_0}\in \Univ_0 \cong \Univ(\g, f_2)$, the formula
    $$ \left(X \bmod \widehat{I_0}\right) \cdot v \defeq X \cdot v $$ gives a well-defined left action of $\Univ_0$ on $\Whi_0(V)$.
\end{proof}

This defines a functor $ \Whi_0 : \whcat_0 \rightarrow \module{\Univ(\g, f_2)}. $ We want to define an equivalence inverse to this functor. One can check that $Q_0 = \Univ(\g, f_1) / \widehat{I_0}$ is in fact a $(\Univ(\g, f_1), \Univ_0)$-bimodule for the action of $\Univ_0$ by right-multiplication. 

\begin{remark}
    There is an algebra isomorphism
    $$ \operatorname{End}_{\Univ(\g, f_1)}(Q_0)^{\mathsf{op}} \overset{\sim}{\longrightarrow} \Univ_0, \quad \theta \longmapsto \theta\big(1 \bmod \widehat{I_0}\big), $$
    where the supscript ``$\mathsf{op}$'' means that we take the opposite multiplication of the endomorphism algebra. This is an analogue, in reduction by stages, of the isomorphism $\operatorname{End}_{\g}(Q_f)^{\mathsf{op}} \cong \Univ(\g, f)$ \cite[Paragraph 1.3]{gan2002quantization}. In particular, the bimodule structure on $Q_0$ becomes transparent with this identification.
\end{remark}

Whence, for $W \in \module{\Univ(\g, f_2)}$, set:
$$ \Ind_0(W) \defeq Q_0 \otimes_{\Univ_0} W, $$
which is a left $\Univ(\g, f_1)$-module.

\begin{lemma}
	For all $W \in \module{\Univ_0}$, $\Ind_0(W) \in \whcat_0$.
\end{lemma}

\begin{proof}
	Set $\left(X \bmod \widehat{I_0}\right) \otimes w \in \Ind_0(W)$. For all $y \in \m_0$, one has:
	\begin{align*}
		(y - \chi_2(y)) \cdot \left(\left(X \bmod \widehat{I_0}\right) \otimes w\right) 
		& =  \left((y - \chi_2(y))X \bmod \widehat{I_0}\right) \otimes w \\
		& = \left(\ad(y) X\bmod \widehat{I_0}\right) \otimes w, 
	\end{align*}
	where the last equality holds because $ X (y - \chi_2(y)) \in \widehat{I_0}. $ Because $y$ is nilpotent, for $n > 0$ large enough, we have
	$$ (y - \chi_2(y))^n \cdot \left(\left(X \bmod \widehat{I_0}\right) \otimes w\right) = \left(\ad(y)^n X\bmod \widehat{I_0}\right) \otimes w = 0. $$
\end{proof}

We state Main Theorem \ref{main:skryabin-stage}, which is an analogue of the Skryabin equivalence (Theorem \ref{theorem:skryabin}).

\begin{theorem} \label{theorem:skryabin-stage}
	We have an equivalence of categories:
	$$ \begin{tikzcd}[column sep = tiny]
		\whcat_0 \arrow[rr, bend left, "\Whi_0"] & \simeq & \module{\Univ(\g, f_2)}. \arrow[ll, bend left, "\Ind_0"]
	\end{tikzcd} $$
	Moreover, $\Whi_1$ and $\Ind_1$ induce an equivalence of categories $\whcat(\g, f_2, \m_2) \simeq \whcat_0$ by restriction and we have $\Whi_2 = \Whi_0 \circ \Whi_1$.
\end{theorem}

\subsection{Proof of Theorem \ref{theorem:skryabin-stage}}

We use the techniques developed in \cite{gan2002quantization}. We split the proof in two propositions, which give the two natural isomorphisms we need.

\begin{proposition} \label{proposition:equivalence-1}
	For any $W \in \module{\Univ_0}$, the following map is a $\Univ_0$-isomorphism:
	\begin{alignat*}{2}
		\Upsilon :  & \ & W & \longrightarrow \Whi_0(Q_0 \otimes_{\Univ_0} W) \\
		& & w & \longmapsto 1 \otimes w.
	\end{alignat*}
\end{proposition}

\begin{proof}
	We consider the twisted action of $\m_0$ on $Q_0 \otimes _{\Univ_0} W$, that we denote by $\tau$. By definition, for all $y \in \m_0$ and $X \otimes w \in Q_0 \otimes _{\Univ_0} W$:
	$$ \tau(y) (X \otimes w) \defeq ((y - \chi_2(y)) \cdot X) \otimes w. $$
	For this action, $\Whi_0(Q_0 \otimes_{\Univ_0} W)$ corresponds the the level $0$ of the Chevalley-Eilenberg cohomology of $(Q_0 \otimes _{\Univ_0} W, \tau)$:
	$$ \Whi_0(Q_0 \otimes _{\Univ_0} W) = \coH^0((Q_0 \otimes _{\Univ_0} W) \otimes_{\C} \Alt^{\bullet}(\m_0)^*). $$
	
	We want to show that the natural map
	$$ \Upsilon : W \longrightarrow \coH^0((Q_0 \otimes _{\Univ_0} W) \otimes_{\C} \Alt^{\bullet}(\m_0)^*) $$
	is an isomorphism. To do so, we endow both sides with filtrations. 
	
	Since $W$ is finitely generated, there exists a finite dimensional vector space $W_0 \subseteq W$ such that $ W = \Univ_0 \cdot W_0. $	We define the induced Kazhdan filtration by setting for all $n \in \Z$:
	$$ W^{(2)}_n \defeq \left(\Univ_0\right)^{(2)}_n \cdot W_0. $$
	According to Gan and Ginzburg, the fact that $\gr^{(2)} Q_0 \cong \C[{\mu_0}^{-1}(0)]$ is a free $\gr^{(2)} \Univ_0 \cong \C[\Slice_2]$-module (see Theorem \ref{theorem:freeness}) implies that we have an $\m_0$-equivariant graded vector space isomorphism
	$$ \gr^{(2)}_{\bullet} Q_0 \otimes_{\gr^{(2)} \Univ_0} \gr^{(2)}_{\bullet} W \cong \gr^{(2)}_{\bullet}(Q_0 \otimes_{\Univ_0} W) $$
	given by the mapping
	$$ [v]^{(2)}_i \otimes [w]^{(2)}_j \longmapsto [v \otimes w]^{(2)}_{i + j} $$
	for any $v \in \left(Q_0\right)^{(2)}_i$, $w \in W^{(2)}_j$ and $i,j \in \Z$.

	Now, let us use a spectral sequence argument as before. We consider the filtered complex:
	$$ \left(\left(Q_0 \otimes_{\Univ_0} W\right)  \otimes_{\C}\Alt^{\bullet} \m_0 \right)^{(2)}_{\bullet}. $$
	The first pages of the spectral sequence are:
	\begin{align*}
		E^{m,n}_0 & = \gr^{(2)}_m \left(\left(Q_0 \otimes_{\Univ_0} W\right)\otimes_{\C} \Alt^{m + n}(\m_0)^*\right), \\
		E^{m,n}_1 & = \coH^{m+n}\left(\gr^{(2)}_m \left(\left(Q_0 \otimes_{\Univ_0} W\right) \otimes_{\C}  \Alt^{\bullet}(\m_0)^*\right)\right) \\
		& \cong \coH^{m+n}\left(\left(\left(\C[{\mu_0}^{-1}(0)] \otimes_{\gr^{\Kaz^2} \Univ_0} \gr^{\Kaz^2} W\right) \otimes_{\C} \Alt^{\bullet}(\m_0)^*\right)^{(2)}_m\right).
	\end{align*}
	
	Claim \ref{claim:freeness} gives the $\m_0$-equivariant graded isomorphism:
	$$ \C[{\mu_0}^{-1}(0)] \cong \C[M_0] \otimes_{\C} \gr^{(2)} \Univ_0, $$
	and it implies the isomorphism 
	$$ \C[{\mu_0}^{-1}(0)] \otimes_{\gr^{(2)} \Univ_0} \gr^{(2)} W \cong \C[M_0] \otimes_{\C} \gr^{(2)} W. $$
	By similar computations as in the proof of Claim \ref{claim:derham}, we get:
	\begin{align*}
		\coH^{n}&\left(\left(\C[{\mu_0}^{-1}(0)] \otimes_{\gr^{(2)} \Univ_0} \gr^{(2)} W\right) \otimes_{\C} \Alt^{\bullet}(\m_0)^*\right)  \\
		& \quad \cong  \left\{\begin{array}{ll}
			\gr^{(2)} W & \text{if $n = 0$} \\
			\{0\} & \text{if $n \neq 0$},
		\end{array}\right.
	\end{align*}
	because $M_0$ is unipotent.
	
	As previously, the spectral sequence is convergent and collapses, which gives a natural isomorphism
	$$ \gr^{(2)}_m \coH^0\left(\left(Q_0 \otimes_{\Univ_0} W\right) \otimes_{\C} \Alt^{\bullet}(\m_0)^*\right) \cong \coH^0\left(\gr^{(2)}_m \left(Q_0 \otimes_{\C} \Alt^{\bullet}(\m_0)^*\right)\right) = \gr^{(2)}_m W $$
	for all $m \in \Z$. Moreover, the following triangle commutes.
	$$ \begin{tikzcd}
		\gr^{(2)}_{\bullet} W \arrow[rr, equal]  \arrow[rd, "\gr \Upsilon"'] & & \gr^{(2)}_{\bullet} W \\
		& \gr^{(2)}_{\bullet} \coH^0\left(\left(Q_0 \otimes_{\Univ_0} W\right)  \otimes_{\C}\Alt^{\bullet} \m_0 \right)  \arrow[ru, "\sim"'] & 
	\end{tikzcd} $$
	
	Finally we can deduce that $\gr \Upsilon$ is an isomorphism and so is $\Upsilon$.
	
\end{proof}

As in \cite{gan2002quantization}, we can state the following useful lemma.

\begin{lemma} \label{lemma:skryabin}
	For any $W \in \whcat_0$, if $\Whi_0(W) = \{0\}$ then $W = \{0\}$.
\end{lemma}

Now we can establish a second natural isomorphism.

\begin{proposition} \label{proposition:equivalence-2}
	For any $V \in \whcat_0$, the following map is a $\Univ(\g, f_1)$-isomorphism:
	\begin{alignat*}{2}
		\Theta :  & \ & Q_0 \otimes_{\Univ_0} \Whi_0(V) & \longrightarrow V \\
		& & \left(X \bmod \widehat{I_{0}}\right) \otimes v & \longmapsto X \cdot v.
	\end{alignat*}
\end{proposition}

\begin{proof}
	By definition of $\Whi_0(V)$, the map $\Theta_V$ is clearly well-defined. Let us consider the kernel
	$$ V' \defeq \Ker \Theta \in \whcat_0. $$
	It is straightforward that:
	$$ \Whi_0(V') = V' \cap \Whi_0(Q_0 \otimes_{\Univ_0} \Whi_0(V)) = V' \cap \{1 \otimes w ~ : ~ w \in \Whi_0(V)\} $$
	where the last equality is given by Proposition \ref{proposition:equivalence-1}. Then we have $\Whi_0(V') = \{0\} $ because of the definition of $\Theta$. Using Lemma \ref{lemma:skryabin}, we conclude that $V' = \{0\}$: $\Theta$ is injective.
	
	To prove the surjectivity we consider the short exact sequence
	$$ \{0\} \longrightarrow Q_0 \otimes_{\Univ_0} \Whi_0(V) \overset{\Theta}{\longrightarrow} V \longrightarrow V'' \longrightarrow \{0\}, $$
	where $V'' \defeq V / \Theta( Q_0 \otimes_{\Univ_0} \Whi_0(V))$ is the cokernel of $\Theta$.
	It consists of $\m_0$-equivariant morphisms, so we can apply the Chevalley-Eilenberg cohomology functor, which leads to a long exact sequence. According to the proof of Proposition~\ref{proposition:equivalence-1}, for $n > 0$:
	\begin{multline*}
		\gr^{(2)}_{\bullet} \coH^n\left(\left(Q_0 \otimes_{\Univ_0} W\right)  \otimes_{\C}\Alt^{\bullet} \m_0 \right) \\ \cong \coH^{n}\left(\left(\left(\gr^{\Kaz^2} Q_0 \otimes_{\gr^{\Kaz^2} \Univ_0} \gr^{\Kaz^2} W\right) \otimes_{\C} \Alt^{\bullet}(\m_0)^*\right)^{(2)}_{\bullet}\right) = \{0\},
	\end{multline*}  
	so one has
	$$ \coH^n\left(\left(Q_0 \otimes_{\Univ_0} W\right)  \otimes_{\C}\Alt^{\bullet} \m_0 \right) = \{0\}. $$
	
	So the long exact sequence degenerates to the following short exact sequence!
	$$ \{0\} \longrightarrow \Whi_0(Q_0 \otimes_{\Univ_0} \Whi_0(V)) \overset{\Whi_0(\Theta)}{\longrightarrow} \Whi_0(V) \longrightarrow \Whi_0(V'') \longrightarrow \{0\}. $$
	The map $\Whi_0(\Theta)$ is given by
	$$ \Whi_0(Q_0 \otimes_{\Univ_0} \Whi_0(V)) \ni 1 \otimes x \longmapsto x, $$
	hence it is bijective and its cokernel $\Whi_0(V'')$ is $\{0\}$. Hence, Lemma \ref{lemma:skryabin} implies that $V'' = \{0\}$.
	
	Finally, the kernel and cokernel of $\Theta$ are trivial, so $\Theta$ is an isomorphism.	
\end{proof}

Propositions \ref{proposition:equivalence-1} and \ref{proposition:equivalence-2} provide the natural isomorphisms for the equivalence of category $\whcat_0 \simeq \module{\Univ(\g, f_2)}$ stated in Theorem \ref{theorem:skryabin-stage}. To conclude, let us prove the last part of the statement.

By a direct checking, the image of $\whcat(\g, f_2, \m_2)$ by the functor $\Whi_1$ lies in $\whcat_{0}$. Take $W \in \whcat_0 \subseteq \module{\Univ(\g, f_1)}$. We consider $\Ind_1(W) = Q_1 \otimes_{\Univ(\g, f_1)} W$. We already know that $\Ind(W) \in \whcat(f, \g, f_1)$, and for all $y \in \m_0$ and $X \otimes v \otimes \Ind_1(W)$ one has:
$$ (y - \chi_2(y)) \cdot (X \otimes v) = (\ad(y) X) \otimes v + X \otimes ((y - \chi_2(y)) \cdot v). $$
for $n > 0$ large enough,
$$ \ad^n(y) X = 0 \quad \text{and} \quad (y - \chi_2(y))^n \cdot v = 0, $$
hence $(y - \chi_2(y))^n \cdot (X \otimes v) = 0$. Therefore $\Ind_1(W) \in \whcat(f, \g, f_2)$. We have an equivalence since the restriction is well-defined and it is clear that $ \Whi_2 = \Whi_0 \circ \Whi_1. $ This achieves the proof.

\bibliographystyle{plain}
\bibliography{sn-bibliography}

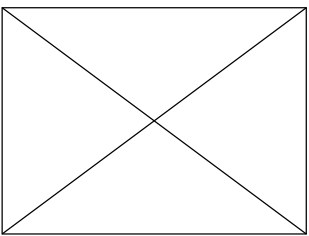
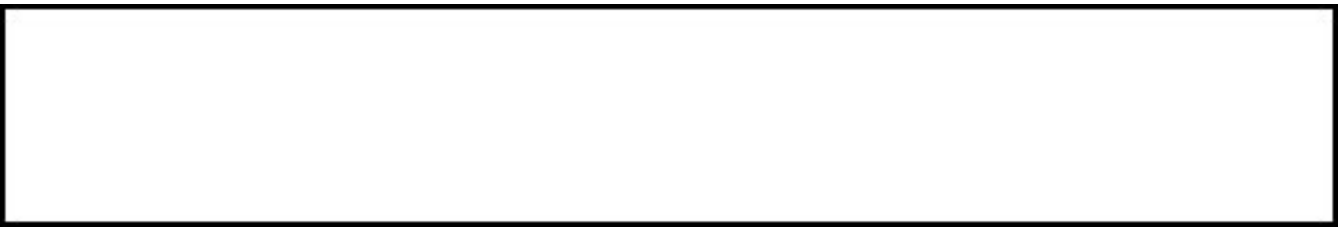
\end{document}